\documentclass[a4paper]{article}

\usepackage[plainpages=false, colorlinks=true, 
            linkcolor=black, urlcolor=black, citecolor=black]{hyperref}
\usepackage{geometry}
\usepackage[utf8]{inputenc}

\usepackage{amsmath, amssymb, amsthm, amsfonts, cite}
\usepackage{color}
\usepackage{booktabs}
\usepackage{tikz}
\usepackage{scrextend}
\usepackage{array}
\usepackage{arydshln}
\usepackage{multirow}
\usepackage[font=scriptsize, labelfont=bf]{caption}
\usepackage{authblk}
\usepackage{cite}

\usepackage{algpseudocode}
\usepackage{graphicx}
\usepackage{array, tabularx}
\usepackage{algorithm, algcompatible}
\usepackage{booktabs} 
\usepackage{paralist} 
\usepackage{verbatim} 
\usepackage{subfig} 
\usepackage{xcolor, marginnote, enumitem}
\usepackage[numbers, sort&compress]{natbib}

\numberwithin{equation}{section}

\theoremstyle{theorem}
\newtheorem{lemma}{Lemma}
\newtheorem{theorem}{Theorem}
\newtheorem{proposition}{Proposition}
\newtheorem{corollary}{Corollary}

\newtheorem{remark}{Remark}

\theoremstyle{definition}
\newtheorem{definition}{Definition}

\DeclareMathOperator{\gph}{gph}
\DeclareMathOperator{\dom}{dom}
\DeclareMathOperator{\dist}{dist}

\DeclareMathOperator*{\argmin}{arg\, min}

\DeclareMathOperator{\Div}{div}

\newcommand{\grad}{\nabla}

\newcommand{\lookUp}[1]{}

\newcommand{\bds}{\boldsymbol}

\newcommand{\mS}{\mathcal{S}}

\algnewcommand\INPUT{\item[\textbf{Input:}]}%
\algnewcommand\OUTPUT{\item[\textbf{Output:}]}%

\DeclareMathAlphabet{\mathbfit}{OML}{cmm}{b}{it}

\newcommand{\highlight}[1]{\colorbox{white!90!black}{\ensuremath{#1}}}

\begin{document}

\title{A Preconditioned Alternating Minimization Framework for Nonconvex and Half Quadratic Optimization}

\author{Shengxiang Deng\thanks{Institute for Mathematical Sciences, 
		Renmin University of China,  China.
		Email: \href{mailto:2018103581@ruc.du.cn}{2018103581@ruc.du.cn}.} \quad 
	Ismail Ben Ayed \thanks{\'Ecole de Technolgie Superieure de Montr\'eal, Canada.
		Email: \href{mailto:ismail.benayed@etsmtl.ca}{ismail.benayed@etsmtl.ca}.}  \quad
	Hongpeng Sun\thanks{Institute for Mathematical Sciences, 
		Renmin University of China, China.	Email: \href{mailto:hpsun@amss.ac.cn}{hpsun@amss.ac.cn}.} }

\maketitle

\begin{abstract}
  For some typical and widely used non-convex half-quadratic regularization models and the Ambrosio-Tortorelli approximate Mumford-Shah model, based on the Kurdyka-\L ojasiewicz analysis and the recent nonconvex proximal algorithms, we developed an efficient preconditioned framework aiming at the linear subproblems that appeared in the nonlinear alternating minimization procedure.  Solving large-scale linear subproblems is always important and challenging for lots of alternating minimization algorithms. By cooperating the efficient and classical preconditioned iterations into the nonlinear and nonconvex optimization, we prove that only one or any finite times preconditioned iterations are needed for the linear subproblems without controlling the error as the usual inexact solvers. 
  The proposed preconditioned framework can provide great flexibility and efficiency for dealing with linear subproblems and guarantee the global convergence of the nonlinear alternating minimization method simultaneously.   
\end{abstract}


\section{Introduction}

The aim of this paper is to develop a preconditioned framework to deal with linear subproblems for the nonlinear and nonconvex alternating minimization algorithms, while applying to some nonconvex half quadratic regularized problems or the Ambrosio-Tortorelli approximate Mumford-Shah model \cite{LT,MD}.  For the half quadratic models, we mainly focus on the truncated quadratic model, the Geman-McClure model, and the Hebert-Leahy model \cite{GR, GY, GM, HL}. The truncated quadratic models including the Geman-Reynolds type model \cite{GR} and Geman-Yang type model \cite{GY} are rooted from the Markov random fields. 
All the half quadratic models addressed here have statistical interpretations and we refer to \cite{BR,HRRS, GW} for more details. 
These nonconvex regularizations are widely used for image restorations, segmentation, stereo, optical flow, and so on, which have vast applications in  medical imaging, computer vision, and inverse problems \cite{BKS,BR, BZ, GW, HSS, MD}.

Due to the importance of these models, there are many theoretical and algorithmic studies on each of these models \cite{AA,BZ,CBAB1,CBAB2,Chamb, CIM, ID, HAD, MN, NN,SC}. For the truncated quadratic model, the graduated non-convexity algorithm was proposed in \cite{BZ} and see \cite{CLMS,MN} for recent developments. Difference of convex algorithm (DCA) \cite{HAD}, preconditioned DCA \cite{DS}, and the first-order primal-dual algorithm \cite{SC} are also developed for the truncated quadratic models. For the German-McClure model, Hebert-Leahy model, and Ambrosio-Tortorelli  model, there are also lots of analysis and algorithmic developments; see \cite{AA, BKS,BR, CBAB1, HSS} and so on.

Now, let's restrict our attention on the alternating minimization (AM) optimization. The alternating minimization algorithms are essentially the same as the majorize-minimize type algorithm including the expectation–maximization (EM) algorithm and the bound optimization \cite{AIG,CIM, DV, CM, LI, RYY, TMBB}. The general alternating minimization method for $L(\bds{u},\bds{y})$ reads as follows
\begin{equation}\label{eq:alter:model}
\inf_{\bds{u} \in X, \bds{y}\in Y} L(\bds{x}, \bds{y}):=f(\bds{u})+Q(\bds{u}, \bds{y})+g(\bds{y}).
\end{equation}
Given initial value $\bds{y}^0$ (or $\bds{x}^0$), iterate for $k=0,1,\cdots$, until some stopping criterion is fulfilled
\begin{subequations}\label{eq:alter:mini:orig}
	\begin{align}
	\bds{u}^{k+1} &\in\argmin_{\bds{u}\in X}L(\bds{u}, \bds{y}^k), \label{update:x:orig}\\
	\bds{y}^{k+1} &\in\argmin_{\bds{y} \in Y} L(\bds{u}^{k+1}, \bds{y}),\label{update:y:orig}
	\end{align}
\end{subequations}
where $f$ and $g$ are proper lower semicontinuous functions, $Q$ is a $C^1$ function with  local Lipschitz continuous gradient $\nabla Q$, and $X$, $Y$ are the corresponding finite dimensional Hilbert spaces. Throughout this paper, we assume all the variables, the spaces, and the operators including the integrals setting are all finite dimensional.  For the half quadratic models, \eqref{eq:alter:model} can be designed for the following equivalent  minimization problem, where the pioneering work can be find in \cite{GR, GY}:
\begin{equation}\label{eq:fun:original}
\inf_{\bds{u} \in X} F(\bds{u}), \quad  F(\bds{u}):=D(\bds{u}) + R(\bds{u})  = \inf_{\bds{y} \in Y}L(\bds{u},\bds{y}).
\end{equation}
Here $D(\bds{u})$ is the data term and $R(\bds{u})$ is the regularization term. 
Instead of minimizing the original functional \eqref{eq:fun:original}, it is convenient to minimizing \eqref{eq:alter:model} with the auxiliary variable $\bds{y}$. In discrete optimization including the binary optimization, the auxiliary variable $\bds{y}$ is called ``line process" in \cite{BZ, GR, GY} or more general ``outlier process" \cite{BR}. These line or outlier processes are popular and widely used, not only because they  have physical or statistic intuition and the ability to model spatial properties of discontinuities \cite{BR} but also they can make the optimization procedure more efficient and stable. Actually, one can get \eqref{eq:alter:model} from the  original \eqref{eq:fun:original} by the Fenchel-Rockafellar duality theory. We refer to the Geman-Reynolds model \cite{GR} and Geman-Yang model \cite{GY} which are two typical and different duality processes  to get \eqref{eq:alter:model} from \eqref{eq:fun:original}.  In \cite{AA}, the alternating minimization focused on convex regularization term $R(\bds{u})$ is discussed.
In \cite{CBAB1,CBAB2, AIG} and the thesis \cite{Chamb}, two alternating minimization algorithms including ``ARTUR" and ``LEGEND" algorithms were developed with global convergence analysis  for convex models with convex regularizations.
Although there are rich studies on alternating minimization for the half quadratic models or the Ambrosio-Tortorelli model, however, solving the large scale linear subproblems for $\bds{u}$ (or $\bds{y}$) is still very challenging. Recent attempts can be found in \cite{CIM,LI} where conjugate gradient method with line search as an inexact solver was developed.  Based on the weighted quadratic approximations on $F(\bds{u})$, truncated conjugate gradient method with error control was also developed for $C^1$ function $F(\bds{u})$ in \cite{RYY}. 

Inspired by the novel preconditioning techniques developed for linear subproblems that appeared in nonlinear convex optimizations \cite{BS1,BS2}, the new development of Kurdyka-\L ojasiewicz (KL) analysis, and the proximal  nonconvex optimization \cite{AB,ABRC,ABS, BST}, we proposed a preconditioned framework for the nonconvex alternating minimization  aiming at dealing with the linear problems  especially large-scale problems efficiently for \eqref{eq:alter:model}. 
Our contributions belong to
the following folds. First, we proposed a preconditioned framework that can deal with any linear subequation of $\bds{u}$ or $\bds{y}$. Any finite number of preconditioned iterations can guarantee the convergence of the whole nonlinear alternating minimization algorithm without solving the linear subequations with middle or high accuracy, which is different from the inexact solvers through error control. Especially, for the Ambrosio-Tortorelli model, preconditioned iterations can be employed for both linear equations involving $\bds{u}$ and $\bds{y}$. Second, with the analysis of the semialgebraic sets and o-minimal structure, we prove that all the functions $L(\bds{u},\bds{y})$ of discussed models are KL functions. Together with the boundedness of the iteration sequence, we obtain the global convergence of the proposed preconditioned alternating minimization by the recent developments of proximal minimization algorithms \cite{AB,ABRC,ABS, BST}. Third, our preconditioned framework can deal with the nonsmooth and nonconvex truncated quadratic problems efficiently with global convergence, while they are precluded by convex assumptions as in \cite{AA, AIG,CBAB2}  or $C^1$ conditions as in \cite{RYY}. Fourth, we developed several efficient red-black Gauss-Seidel preconditioners for both isotropic and anisotropic  equations in divergence form by finite difference method \cite{LZ, WE}. Numerical tests show that one can get lower energy and better reconstructions more efficiently with the proposed preconditioned framework compared with solving the linear system of \eqref{eq:alter:mini:orig} with moderate accuracy and  without any proximal terms.

The rest of this paper is organized as follows. In section \ref{sec:KL:models}, we give a brief introduction to the half quadratic models, the Ambrosio-Tortorelli model, and the basic KL analysis. In section \ref{sec:pre:KL}, we give an illustration of the motivation of our preconditioned framework and prove that all $L(\bds{u},\bds{y})$ of discussed models are KL functions.  In section \ref{sec:convergence}, we first prove the iterative sequence $\{(\bds{u}^k,\bds{y}^k)\}$ is bounded for each model. We then get the global convergence with the KL properties by \cite{AB,ABRC,ABS, BST}. 
In section \ref{sec:numer}, we first give the detailed five-point stencils of the symmetric Gauss-Seidel iterations along with the preconditioners and the numerical experiments then follows.  In  section \ref{sec:conclude}, we give some discussions and conclusions.

\section{Nonconvex half quadratic models and KL functions}\label{sec:KL:models}

For the half quadratic models, let's begin with the following truncated quadratic regularization with gradient operator $\nabla :=[\nabla_1,\nabla_2]$ for image restoration \cite{BZ,BR,MD, MN}
\begin{align}
&\inf_{\bds{u} \in X} F(\bds{u} ) =D(\bds{u}) + \int_{\Omega}P^I(\bds{u} )\mathrm{d}\sigma ,\quad P^I(\bds{u} ):=\frac{\mu}{2}\min(|\nabla \bds{u} |^2,\frac{\lambda}{\mu }), \label{equation:iso}  \\
&\inf_{\bds{u}  \in X} F(\bds{u} ) = D(\bds{u}) + \int_{\Omega}P^A(\bds{u})\mathrm{d}\sigma, \quad P^A(\bds{u}):=\frac{\mu}{2}\sum_{i=1}^2\min(|\nabla_i \bds{u} |^2, \frac{\lambda}{\mu }), \label{equation:ani} 
\end{align}
where $\lambda$ and $\mu$ are positive constants, $X$ is finite dimensional discrete image space and henceforth $D(\bds{u})$ is the following data term  with $A$ being a linear and bounded operator 
\[
D(\bds{u}): = {\|A\bds{u}  -\bds{u} _0\|_2^2}/{2}.
\]
$P^I$ or $P^A$ is the isotropic or anisotropic truncated regularizations. Let's introduce the image space $X$ and the dual spaces $Y$ and $Z$, i.e.,
\begin{equation}
X =  \{\bds{u} : \Omega \rightarrow \mathbb{R} \}, \quad Y = \{\bds{u} : \Omega \rightarrow \mathbb{R}^2\}, \quad Z = \{\bds{u}  :\Omega \rightarrow \mathbb{R}\},
\end{equation}
where $\Omega$ is the image domain.
The truncated quadratic model \eqref{equation:iso} can be reformulated as the following Geman-Reynolds model \cite{GR},
\begin{equation}\label{eq:b;truncated:qua}
\inf_{\bds{u} \in X, \bds{b} \in Z}  L_{GR}(\bds{u} , \bds{b} ) : =   D(\bds{u})+\frac{\lambda}{2} \int_{\Omega}\left[ \bds{b}(\frac{\mu}{\lambda} |\nabla \bds{u} |^2- 1) + I_{\{0\leq \bds{b}  \leq 1\}}(\bds{b} ) +1\right] \mathrm{d}\sigma, 
\end{equation}
or the following Geman-Yang model \cite{GY,MD}: $ \inf_{\bds{u}  \in X, \bds{l} \in Y}L_{GY}(\bds{u}, \bds{l} )$, $\bds{l} =(\bds{l}_1,\bds{l}_2)^T$ with $L_{GY}(\bds{u} , \bds{l} )$ defined by, 
\begin{equation}\label{eq:L;truncated:qua}
D(\bds{u}) + \frac{\mu}{2} \|\nabla \bds{u} -\bds{l} \|_2^2 +  \mu \int_{\Omega}H(\bds{l} ;\frac{\lambda}{\mu}) \mathrm{d}\sigma, \  H(\bds{l} ;\frac{\lambda}{\mu}): =\begin{cases}
\sqrt{\frac{\lambda}{\mu}} |\bds{l} | -\frac{|\bds{l} |^2}{2}, \ |\bds{l} | \leq  \sqrt{\frac{\lambda}{\mu}},  \\
\frac{\lambda}{2\mu}, \ \quad |\bds{l} | >  \sqrt{\frac{\lambda}{\mu}}.
\end{cases}
\end{equation}
For the anisotropic case, similarly, we can reformulate \eqref{equation:ani} as 
\begin{equation}\label{eq:b;truncated:qua:ani}
\inf_{\bds{u}  \in X, \bds{b} \in Y}   L_{GR}^A(\bds{u} , \bds{b} ): =  D(\bds{u}) +\frac{\lambda}{2}\sum_{i=1}^2\int_{\Omega}\left[ \bds{b}_i(\frac{\mu}{\lambda} |\nabla_i \bds{u} |^2- 1) + I_{\{0\leq \bds{b}_i \leq 1\}}(\bds{b}_i)+1\right]\mathrm{d}\sigma, 
\end{equation}
and
\begin{equation}\label{eq:L;truncated:qua:ani}
\inf_{\bds{u}  \in X, \bds{l} \in Y}L_{GY}^A(\bds{u} , \bds{l} ):=D(\bds{u})+\sum_{i=1}^2 \big(\frac{\mu}{2} \|\nabla_i \bds{u} -\bds{l}_i\|_2^2 +  \mu \int_{\Omega}H(\bds{l}_i; \frac{\lambda}{\mu})\mathrm{d}\sigma\big).
\end{equation}
The following German-McClure model is also widely used \cite{GM,GR}
\begin{equation}\label{eq:gm:original}
\inf_{\bds{u} \in X} F_{GM}(\bds{u} ):=D(\bds{u})+ \frac{\mu}{2} \int_{\Omega}\frac{|\nabla \bds{u} |^2/\lambda}{1+ |\nabla \bds{u} |^2/\lambda}\mathrm{d}\sigma,
\end{equation}
which is equivalent to the following minimization problem \cite{BR, Chamb}
\begin{equation}\label{eq:gm:b}
\inf_{\bds{u}  \in X, \bds{b} \in Z}L_{GM}(\bds{u} , \bds{b} ): = D(\bds{u})+ \frac{\mu}{2} \int_{\Omega}\left[  \bds{b}  \frac{|\nabla \bds{u} |^2}{\lambda} +\bds{b} -2\sqrt{\bds{b} }+1  \right]\mathrm{d}\sigma, \quad \bds{b}  \geq 0.
\end{equation}
Similarly,  the anisotropic German-McClure model can be
\begin{equation}\label{eq:gm:b:ani}
\inf_{\bds{u}  \in X, \bds{b} \in Y}L_{GM}(\bds{u} , \bds{b} ): = D(\bds{u})+ \sum_{i=1}^2\frac{\mu}{2} \int_{\Omega}\left[  \bds{b}_i \frac{|\nabla_i \bds{u} |^2}{\lambda} +\bds{b}_i-2\sqrt{\bds{b}_i}+1  \right]\mathrm{d}\sigma,
\end{equation}
where $ \bds{b} =( \bds{b}_1,\bds{b}_2)$ and $\bds{b}_i \geq 0$, $i=1,2$. Now, let's turn to the Hebert-Leahy model which reads as follows \cite{HL}
\begin{equation}\label{eq:HL:orignal}
\inf_{\bds{u} \in X} F_{HL}(\bds{u} ):=D(\bds{u}) + \frac{\mu}{2} \int_{\Omega} \log(1+\frac{|\nabla \bds{u} |^2}{\lambda})\mathrm{d}\sigma, 
\end{equation}
which is equivalent to the following minimization problem \cite{BKS, BR, Chamb}
\begin{equation}\label{eq:hl:b}
\inf_{\bds{u}  \in X, \bds{b} \in Z}L_{HL}(\bds{u} , \bds{b} ):=D(\bds{u})+\frac{\mu}{2} \int_{\Omega}(\bds{b} \cdot\frac{|\nabla \bds{u} |^2}{\lambda}+\bds{b}  - \log(\bds{b} ) -1)\mathrm{d}\sigma, \quad \bds{b} >0.
\end{equation}
Similarly, the anisotropic Hebert-Leahy model can be
\begin{equation}\label{eq:hl:b:ani}
\inf_{\bds{u}  \in X, \bds{b} \in Y}L_{HL}(\bds{u} , \bds{b} ):=D(\bds{u})+\sum_{i=1}^2\frac{\mu}{2} \int_{\Omega}(\bds{b}_i\cdot\frac{|\nabla_i \bds{u} |^2}{\lambda}+\bds{b}_i - \log(\bds{b}_i) -1)\mathrm{d}\sigma,
\end{equation}
where $ \bds{b} =( \bds{b}_1,\bds{b}_2)$ and $\bds{b}_i \geq 0$, $i=1,2$. 
We will also discuss the following Ambrosio–Tortorelli approximation of the Mumford-Shah model \cite{LT}
\begin{align}
\inf_{\bds{u}  \in X, \bds{s}  \in Z} L_{MS}(\bds{u} ,\bds{s} ):=& \frac{\|\bds{u} -\bds{u}_0\|_{2}^2}{2} + \alpha  \int_{\Omega} \bds{s}^2 |\nabla \bds{u} |^2 \mathrm{d}\sigma \notag\\
&+  \lambda \left( \varepsilon \int_{\Omega} |\nabla \bds{s} |^2 \mathrm{d}\sigma+ \int_{\Omega} \frac{1}{4\varepsilon} |\bds{s} -1|^2 \mathrm{d}\sigma\right). \label{eq:ms:b}
\end{align}

Based on difference of convex functions, we will give another interpretation of the equivalence of the truncated model \eqref{equation:iso} and \eqref{eq:L;truncated:qua} compared to \cite{GY,MD}. The equivalence of the anisotropic cases \eqref{equation:ani} and \eqref{eq:L;truncated:qua:ani} is similar and  omitted here. Let's begin with the following lemma.
\begin{lemma}\label{lem:dual:max}
	Denote $h_a(\bds{t} ) := \max(|\bds{t}|^2/{2}, {a}/{2})$ which is a convex function with positive constant $a$ and $\bds{t}\in Y$. The Fenchel dual function of $h_a(\bds{t} )$  is
	\begin{equation}\label{eq:dual:max:func}
	h_a^*(\bds{l} ) = \begin{cases}
	\sqrt{a}|\bds{l} | -a/2, \quad |\bds{l} | \leq \sqrt{a}, \\
	{|\bds{l} |^2}/{2}, \quad |\bds{l} | >  \sqrt{a}. 
	\end{cases}
	\end{equation}
\end{lemma}
\begin{proof}
	It can be checked that $|\bds{l}||\bds{t}| \leq h_a(\bds{t} ) $ when $|\bds{l}| \leq \sqrt{a}$, since the curve $|\bds{l}||\bds{t}|$ is below $\max(|\bds{t}|^2/{2}, {a}/{2})$ as functions of $|\bds{t}|$ for fixed $\bds{l}$. While $|\bds{l}| \leq \sqrt{a}$, we thus have
	\begin{align*}
	\sup_{\bds{t} \in Y }&\left\{\langle \bds{t},  \bds{l}\rangle -\max({|\bds{t}|^2}/{2}, {a}/{2}) \right\}= 	\sup_{|\bds{t}|\in [0, \sqrt{a}]} \left\{\langle \bds{t},  \bds{l}\rangle -\max({|\bds{t}|^2}/{2}, {a}/{2})\right\} \\
	& = \sup_{|\bds{t}|\in [0, \sqrt{a}]} \left\{\langle \bds{t},  \bds{l}\rangle -{a}/{2}\right\} = \sup_{|\bds{t}| = \sqrt{a}} \left\{\langle \bds{t},  \bds{l}\rangle -{a}/{2}\right\}= |\bds{l}|\sqrt{a} -{a}/{2}. 
	\end{align*}
	While $|\bds{l}| > \sqrt{a}$, it can be verified that $|\bds{l}||\bds{t}|=h_a(t)$  while $|\bds{t}| = \frac{a}{2|\bds{l}|}$ or $2|\bds{l}|$. We thus conclude that, while $|\bds{l}| > \sqrt{a}$, 
	\begin{align*}
	\sup_{\bds{t}\in Y}&\left\{\langle \bds{t},  \bds{l}\rangle -\max({|\bds{t}|^2}/{2}, {a}/{2}) \right\}= 	\sup_{|\bds{t}|\in [0, 2|\bds{l}|]} \left\{\langle \bds{t},  \bds{l}\rangle -\max({|\bds{t}|^2}/{2}, {a}/{2})\right\} \\
	& = \sup_{|\bds{t}|\in [\sqrt{a}, 2|\bds{l}|]} \left\{\langle \bds{t},  \bds{l}\rangle -{|\bds{t}|^2}/{2}\right\} = \left\{\langle \bds{t},  \bds{l}\rangle -{|\bds{t}|^2}/{2}\right\}|_{{|\bds{t}| = |\bds{l}|} }=  {|\bds{l}|^2}/{2}. 
	\end{align*}
	With the discussion of the two cases above, we get \eqref{eq:dual:max:func}. 	
\end{proof}

\begin{lemma}
	The truncated models \eqref{equation:iso} and \eqref{eq:L;truncated:qua} are equivalent. 
\end{lemma}
\begin{proof}
	For  \eqref{equation:iso},  it can be written as the difference of convex functions (DCA): $ F(\bds{u} ) = F_1(\bds{u})-F_2(\bds{u})$ with convex $F_1$ and $F_2$
	\begin{equation}\label{eq:l:minus:max}
	F_1(\bds{u}):= D(\bds{u})+  \int_{\Omega}(\frac{\mu}{2}|\nabla \bds{u} |^2 +\frac{\lambda}{\mu })d\sigma, \quad F_2(\bds{u}) =  \int_{\Omega}\frac{\mu}{2}\max(|\nabla \bds{u} |^2, \frac{\lambda}{\mu })\mathrm{d}\sigma. 
	\end{equation}
	By Lemma \ref{lem:dual:max}, we have
	\begin{equation}\label{eq:max:em:l1}
	\frac{\mu}{2}\max(|\nabla \bds{u} |^2, \frac{\lambda}{\mu }) = {\mu}\max(|\nabla \bds{u} |^2/2, \frac{\lambda}{\mu }/2) =\mu \sup_{\bds{l} \in Y} \left\{\langle \nabla \bds{u} , \bds{l}  \rangle - h_{\frac{\lambda}{\mu }}^*(\bds{l} )\right\}.
	\end{equation}
	Substituting  \eqref{eq:max:em:l1} into \eqref{eq:l:minus:max}, we get this lemma with direct calculation. 
\end{proof}

Henceforth, we will vectorize all the image variables $\bds{u} $, $\bds{u}_0$ and the auxiliary variable $\bds{y}$ including $\bds{l}$, $\bds{b}$ and $\bds{s}$ as $u\in \mathbb{R}^{MN}$, $u_0\in \mathbb{R}^{M_0N_0}$ and $y$ including ${l}$, ${b}$ and ${s}$  in the corresponding spaces and for the corresponding models. Besides, we will still use the linear operators $A$, $\nabla_i$, $i=1,2$ as their corresponding matrix versions after vectorizing all the variables. Now, we will touch some necessary tools from convex and variational analysis \cite{MB,  CL, Roc1}.
The {graph} of a multivalued mapping $F : R^n \Rightarrow R^m$ is defined by
\[
\gph F := \{(x,y) \in R^n \times R^m : y \in F(x)\}, 
\]
whose {domain} is defined by $ \dom F: = \{ x \ | \ F(x) \neq \emptyset \} $. 
Similarly the graph of an extended real-valued function $f :R^n \to R \cup \{+\infty\}$ is defined by 
\[
\gph f := \{(x,s) \in R^n\times R : s = f(x)\}. 
\]
Let $h: \mathbb{R}^n \rightarrow \mathbb{R}\cup \{+\infty\}$ be a proper lower semicontinuous function. Denote $\dom h: = \{ x \in\mathbb{R}^n: \ h(x) < +\infty  \}$. 
For each $x\in\dom f$, 
the limiting-subdifferential of $h$ at $x\in \mathbb{R}^n$, written $\partial f$, is defined as follows \cite{MB, Roc1}, 
\begin{align}
\partial h(x): = \bigg\{ & \xi \in \mathbb{R}^n: \exists x_n \rightarrow x,  h(x_n) \rightarrow h(x),  \xi_n  \rightarrow \xi,  \notag   \\
&\lim_{y\rightarrow x } \inf_{y \neq x_n}\frac{h(y)-h(x_n) - \langle \xi_n, y-x_n \rangle }{|y-x_n|}\geq 0 \bigg\}.  \notag
\end{align}
It is known that the above subdifferential $\partial h$  reduces to the classical subdifferential in convex analysis when $h$ is convex. It can be seen that a necessary condition for $x \in \mathbb{R}^n$ to be a minimizer of $h$ is
$0 \in \partial h$ \cite{AB}. 

For the global and local convergence analysis, we also need the Kurdyka-\L ojasiewicz (KL) property and KL exponent. While the KL properties can help obtain the global convergence of iterative sequences,  the KL exponent can help provide a local convergence rate. 
\begin{definition}[KL property and KL exponent] \label{def:KL}
	A proper closed function $h$ is said to satisfy the KL property at $\bar x \in \dom \partial h$ if there exists $a \in (0,+\infty]$, a neighborhood $\mathcal{O}$ of $\bar x$, and a continuous concave function $\psi: [0,a) \rightarrow (0,+\infty)$ with $\psi(0)=0$ such that:
	\begin{itemize}
		\item [\emph{{(i)}}] $\psi$ is continuous differentiable on $(0,a)$ with $\psi'>0$. 
		\item [\emph{{(ii)}}] For any $x \in \mathcal{O}$ with $h(\bar x) < h(x) <h(\bar x) + a$, one has
		\begin{equation}\label{eq:kl:def}
		\psi'(h(x)-h(\bar x)) \dist(0,\partial h(x)) \geq 1. 
		\end{equation}
	\end{itemize}
	A proper closed function $h$ satisfying the KL property at all points in $\dom \partial h$ is called a KL function. If $\psi$ in \eqref{eq:kl:def} can be chosen as $\psi(s) = cs^{1-\theta}$ for some $\theta \in [0,1)$ and $c>0$, we say that $h$ satisfies KL properties at $\bar x$ with exponent $\theta$. This means that for some $\bar c >0$, we have
	\begin{equation}\label{eq:KL:exponent:theta:exam}
	\dist(0,\partial h(x)) \geq \bar c (h(x)-h(\bar x))^{\theta}. 
	\end{equation}
	If $h$ satisfies KL property with exponent $\theta \in [0,1)$ at all the points of $\dom \partial h$, we call $h$ is a KL function with exponent $\theta$. 
\end{definition}

For KL functions, the semialgebraic functions and definable functions in an o-minimal structure provide a vast field of applications including the KL analysis for our models to be discussed. 
\begin{definition}[Semialgebraic set and Semialgebraic function \cite{ABS}]\label{definition:semialgebraic}
	A subset $S$ of $\mathbb{R}^n$ is called a real semialgebraic set if there exists a finite number of real polynomial functions $P_{i,j}, \ Q_{i,j}: \mathbb{R}^n \rightarrow \mathbb{R}$, such that 
	\[
	S = \bigcup_{j=1}^p \bigcap_{j=1}^q\{ x \in \mathbb{R}^n: P_{i,j} =0, \ Q_{i,j} >0\}.
	\]
	A function $ f:\mathbb{R}^n\to\mathbb{R}\cup\{+\infty\}$ is semialgebraic if its graph is a semialgebraic set of $\mathbb{R}^{n+1}$. 
\end{definition}
A very useful conclusion is that a semialgebraic function has the KL property with $\psi(s)=cs^{1-\theta}$ for some $\theta\in\left[ 0,1\right) \cap \mathbb{Q}$ and $ c>0$, which can be seen as a corollary of Theorem 3.2 of \cite{BDL}. The following Tarski-Seidenberg theorem is very useful for the  analysis of  KL properties.

\begin{theorem}[Tarski-Seidenberg \cite{ABS}]\label{Tarski-Seidenberg}
	Let $S$ be a semialgebraic set in $\mathbb{R}^{m+n}$, then
	\[ 
	\bar{S}:=\{x\in\mathbb{R}^m:(x,y)\in S \   \text{for some}  \  y\in\mathbb{R}^n\}
	\]    
	is a semialgebraic set.
\end{theorem}
For lots of cases, when the Tarski-Seidenberg theorem is not applicable, we can turn to the o-minimal structure which is originated from real algebraic geometry and can cover more complicated cases. We refer to \cite{Dr,ABRC} for its definition. Verifying the 
o-minimal structure directly with the definition  is much more complicated compared to verifying the semialgebraic set and we will focus on the existed results that can be employed directly. 
Let $\Delta$ be an o-minimal structure. A set $A$ is called definable if $A\in\Delta$. A map $f$ is said to be definable if its graph $\gph f \subseteq \mathbb{R}^{m+n}$ is definable \cite{Dr}. Due to their dramatic impacts, these structures are being studied extensively. One of the interests of such structures in optimization is due to the following nonsmooth extension of KL property \cite{ABRC,BDLS} (see Theorem 11 of \cite{BDLS}). 
\begin{theorem}[\cite{BDLS}]\label{o-minimal:KL}
	Any proper lower semicontinuous function $f:\mathbb{R}^n\Rightarrow\mathbb{R}\cup \{+\infty\}$ that is definable in an o-minimal structure $\Delta$ has the KL property at each point of $\dom \partial f$. Moreover, the function $\psi$ is definable in $\Delta$. 
\end{theorem}
Let $R_{\exp} = (R,+,\cdot,\exp)$. Wilkie proved that $R_{\exp}$ is
model complete \cite{AW}. As a direct consequence of this theorem, each definable sets in $R_{\exp}$ is the image of the zero set of a function in $P(x,y,e^x,e^y)=0$ under a natural projection (see page 3 of \cite{Dr} or \cite{Loi}). 
Then by a Khovanskii result on fewnomials \cite{KA}, $R_{\exp}$ is an o-minimal structure. An analytic proof of Wilkie’s theorem is given in \cite{MP}. 
\begin{theorem}[\cite{Dr,Loi}]\label{R_exp:o-minimal}
	The images in $\mathbb{R}^n$ for $n=0,1,2, \ldots$ under projection maps $\mathbb{R}^{n+k}\Rightarrow\mathbb{R}^n$ of sets with the form $\{(x,y)\in\mathbb{R}^{n+k}:P(x,y,e^x,e^y)=0\}$ is definable, where $P$ is a real polynomial in $ 2(n+k)$ variables with $x:=(x^1, \ldots,x^n)$, $y:=(y^1,\ldots,y^k)$, $e^x:=(e^{x^1}, \ldots,e^{x^n})$ and $e^y:=(e^{y^1},\ldots,e^{y^k})$.
\end{theorem}

\section{Preconditioned framework and KL properties}\label{sec:pre:KL}
In this section, we will investigate the following preconditioned  alternating minimization framework  \eqref{eq:general:al:min} and give an illustration of our motivation  through the 
Lemma \ref{lem:feaible:percon} to be discussed. The proposed framework includes the updates of $u^{k+1}$ in \eqref{eq:b;truncated:qua}, \eqref{eq:L;truncated:qua}, \eqref{eq:gm:b}, \eqref{eq:hl:b}, \eqref{eq:ms:b} for the isotropic cases and \eqref{eq:b;truncated:qua:ani}, \eqref{eq:L;truncated:qua:ani}, \eqref{eq:gm:b:ani}, \eqref{eq:hl:b:ani} for the anisotropic cases along with the update  $s^{k+1}$ in \eqref{eq:ms:b}. 
Our preconditioned framework for \eqref{eq:alter:model} is as follows:
\begin{subequations}\label{eq:general:al:min}
	\begin{align}
	u^{k+1} &\in\argmin{L(u, y^k)+\frac{1}{2}\|u-u^k\|_{M_k}^2},\label{update:x}\\
	y^{k+1} &\in\argmin{L(u^{k+1}, y)+\frac{1}{2}\|y-y^k\|_{N_k}^2},\label{update:y}
	\end{align}
\end{subequations}
where $M_k, N_k$ are the proximal matrices satisfying \cite{ABRC, ABS}
\begin{equation}\label{eq:M:N:condition}
0<    \gamma_{-} I \leq M_k \leq  \gamma_{+} I, \quad 0<    \mu_{-} I \leq N_k \leq  \mu_{+} I, \quad  0<\gamma_{-}, \ \mu_{-}<\gamma_{+}, \ \mu_{+}  < +\infty.
\end{equation}
Here $y^k$ denotes $b^k$, $l^k$ or $s^k$ in the corresponding models.
We will employ the metric induced by $M_k$ depending on the corresponding model, which turns out to be the classical and powerful preconditioned iterations. It can bring out flexibility and efficiency for linear subproblems.

Henceforth we will focus on the denoising problems with $A=I$. We can reformulate the original Euler-Lagrangian  equation  for $u^{k+1}$ (or $s^{k+1}$ for \eqref{eq:ms:b}) as the following general form 
\begin{equation}\label{eq:original:linear:sys}
\mathbb{T}_k u = \mathfrak{b}^k  \ \text{in} \  \Omega,  \quad \frac{\partial u^{k+1}}{\partial \nu}|_{\partial \Omega} = 0,  \quad \mathbb{T}_k:=  \gamma^k(x)I+ \nabla^* \mathcal{B}^k \nabla, 
\end{equation}
where $\gamma_k(x) \geq \bar \gamma $ for arbitrary $x\in \Omega$ with the constant $\bar \gamma >0$ and the linear operator $\mathcal{B}^k$  is positive semidefinite.
We have $\mathbb{T}_k = I - \mu \Delta$, $\mathfrak{b}^k = u_0 + \nabla^*l^k$ with $\gamma^k(x) = 1.0$ and $\mathcal{B}^k = \text{Diag}[\mu I,\mu I]$ for \eqref{eq:L;truncated:qua},   $\mathbb{T}_k = I + \mu \nabla^* {b}^k \nabla$, $\mathfrak{b}^k = u_0$ with $\gamma^k(x) = 1.0$ and $\mathcal{B}^k = \text{Diag}[\mu {b}^k,\mu {b}^k]$ for \eqref{eq:b;truncated:qua} and $\mathfrak{b}^k = u_0$ with $\gamma^k(x)=1.0$
\begin{equation}\label{eq:bk:general}
\mathbb{T}_k = I +  \nabla^*(  \mathcal{B}^k \nabla), \quad \mathcal{B}^k : = \text{Diag}[\mu/\lambda b_1^k ,  \mu/\lambda b_2^k], 
\end{equation}
for \eqref{eq:gm:b:ani}
and \eqref{eq:hl:b:ani}. Similarly, for $s^{k+1}$ update in \eqref{eq:ms:b},  we have $\mathbb{T}_k =  (\frac{\lambda}{2\varepsilon} + 2 \alpha|\nabla u^k|^2)I -   \lambda\varepsilon \Delta$ with $\gamma^k(x) = \frac{\lambda}{2\varepsilon} + 2 \alpha |\nabla u^k|^2$ and $ \mathcal{B}^k = \text{Diag}[\alpha \varepsilon,\alpha \varepsilon]$. All the other cases are similar. Inspired by the recent development of preconditioning technique for linear subproblems in the nonlinear convex \cite{BS1,BS2,BS3} or nonconvex iteration \cite{DS}, we will introduce the classical preconditioned iteration to deal with linear subproblems \eqref{eq:original:linear:sys}. Our motivation mainly comes from the following lemma, i.e., Lemma \ref{lem:feaible:percon}.

\begin{lemma}\label{lem:feaible:percon}
	With appropriately chosen linear operators $M_k$ and $N_k$ satisfying \eqref{eq:M:N:condition}, for the update of $u^{k+1}$ in the isotropic cases \eqref{eq:b;truncated:qua}, \eqref{eq:L;truncated:qua}, \eqref{eq:gm:b}, \eqref{eq:hl:b}, \eqref{eq:ms:b} and the anisotropic cases \eqref{eq:b;truncated:qua:ani}, \eqref{eq:L;truncated:qua:ani}, \eqref{eq:gm:b:ani}, \eqref{eq:hl:b:ani} along with $s^{k+1}$ in \eqref{eq:ms:b}, these updates can be written as the following classical preconditioned iteration for the original equation \eqref{eq:original:linear:sys},
	\begin{equation}\label{eq:pre:iter}
	u^{k+1}: = u^k + \mathbb{M}_k^{-1}[\mathfrak{b}^k -\mathbb{T}_k u^k], \quad \mathbb{M}_k: = \mathbb{T}_k + M_k.
	\end{equation}
\end{lemma}
\begin{proof}
	Suppose the original Euler-Lagrangian equation for the corresponding functional is \eqref{eq:original:linear:sys}. With adding proximal term in \eqref{update:x}, the Euler-Lagrangian equation for $u$ becomes
	\begin{equation}
	\mathbb{T}_k u +M_k(u-u^k)  - \mathfrak{b}^k=0. 
	\end{equation}
	Still with notation $u^{k+1}$ as the solved $u$ above, we thus arrive at
	\begin{subequations}\label{eq:pre:iterations}
		\begin{align}
		u^{k+1} &= (M_k+ \mathbb{T}_k)^{-1}(\mathfrak{b}^k + M_ku^k)   \\  
		& = (M_k+ \mathbb{T}_k)^{-1}[(M_k+ \mathbb{T}_k) u^k + \mathfrak{b}^k - \mathbb{T}_k u^k]   \\
		&=u^k + \mathbb{M}_k^{-1}[\mathfrak{b}^k  - \mathbb{T}_ku^k],      
		\end{align}
	\end{subequations}	
	which is essentially the classical preconditioned iteration for solving the linear equation $\mathbb{T}_k u=\mathfrak{b}^k$ \cite{SA}.
	We thus reformulate the proximal iteration in \eqref{eq:general:al:min} as the preconditioned iteration \eqref{eq:pre:iter}, which will turn out very useful.
\end{proof}	
Through Lemma \ref{lem:feaible:percon}, we introduce the classical preconditioned iterations from numerical linear algebra and computation to the nonlinear alternating minimization by specially designed positive definite proximal terms $M_k$ and $N_k$. The idea  can also be found in \cite{BS3,DS}. Here, for self-completeness, we give some explanation through a concrete example. 
\begin{remark}\label{rem:pre:gs}
	Suppose the discretization of the linear operator $\mathbb{T}_k$  in Lemma \ref{lem:feaible:percon} is $D_k-E_k-E_k^*$  where $D_k$ is the diagonal part, $-E_k$ represents the strict lower triangular part and $E_k^*$ is the transpose of $E_k$.  Here we still use $\mathbb{T}_k$ as its corresponding discrete matrix.  If  choosing $ \mathbb{M}_k$ as the symmetric Gauss-Seidel preconditioner involving $ \mathbb{T}_k$, considering the positive definiteness requirement of $M_k$ as in \eqref{eq:M:N:condition}, we can choose  \cite{SA} (chapter 4.1) (or \cite{BS1})
	\begin{equation}\label{eq:pre:perturb}
	\mathbb{M}_k = \mathbb{T}_k + E_k^*D_k^{-1}E_k + \eta I = D_k-E_k-E_k^*+ E_k^*D_k^{-1}E_k + \eta I,
	\end{equation}
	where $\eta >0$ being a tiny positive constant and $I$ being the identity matrix. The small perturbation with $\eta I$ is to guarantee the positive definiteness of $M_k$. Actually, we have the proximal metric 
	\begin{equation}\label{eq:preconditioner}
	M_k = \mathbb{M}_k- \mathbb{T}_k = E_k^*D_k^{-1}E_k + \eta I  \geq \eta I,
	\end{equation}
	which is positive definite. However, we do not need to calculate the explicit form of  $\mathbb{M}_k$ or $M_k$. Instead, we can do it through rewriting \eqref{eq:pre:iterations} as follows 
	\begin{subequations}\label{eq:pre:iterations:tiny}
		\begin{align}
		u^{k+1} & = (M_k+ \mathbb{T}_k)^{-1}[(M_k+ \mathbb{T}_k) u^k + \mathfrak{b}^k + \eta u^k - (\mathbb{T}_k + \eta I) u^k]   \\
		&=u^k + \mathbb{M}_k^{-1}[  \bar{\mathfrak{b}}^k  - \bar {\mathbb{T}}_k u^k],      
		\end{align}
	\end{subequations}
	where $\bar{\mathfrak{b}}^k: =\mathfrak{b}^k+  \eta u^k $ and $\bar {\mathbb{T}}_k = {\mathbb{T}}_k + \eta I$. 
	This means the update \eqref{eq:pre:iterations:tiny}  is exactly the one time symmetric Gauss-Seidel iteration for the linear equation $\bar {\mathbb{T}}_k u = \bar{\mathfrak{b}}^k$, which is also equivalent to one time symmetric Gauss-Seidel iteration for the linear equation $ {\mathbb{T}}_k u = {\mathfrak{b}}^k$ with preconditioner in \eqref{eq:pre:perturb}.

\end{remark}
Furthermore, it is proved that any  finite preconditioned iterations still provide a preconditioner \cite{BS1} satisfying \eqref{eq:M:N:condition}, i.e., 
\begin{equation}\label{eq:preconditioner:multiple}
u^{k+(i+1)/n} = u^{k+i/n} + \mathbb{M}_k^{-1}( \bar{\mathfrak{b}}^k-  \bar{\mathbb{T}}_k  u^{k+i/n}),\quad 
i = 0,\ldots,n-1
\end{equation}
corresponds to $u^{k+1} = u^k + \mathbb{M}_{k,n}^{-1}(\bar{\mathfrak{b}}^k-  \bar{\mathbb{T}}_k  u^{k})$ and we just need to choose large enough $\gamma_{+}$ according to fixed $n$ for meeting the requirement of $\mathbb{M}_{k,n}$ in \eqref{eq:M:N:condition}. 
We thus built a flexible framework for introducing the classical preconditioning techniques. However, how to design efficient preconditioners for the corresponding  $\mathbb{T}_k$ especially the anisotropic cases  is still very subtle and challenging. We leave them to section \ref{sec:numer}.

Now, let's turn to the discussion of the KL-properties of \eqref{eq:b;truncated:qua}, \eqref{eq:L;truncated:qua}, \eqref{eq:gm:b}, \eqref{eq:hl:b}, and \eqref{eq:ms:b}. The anisotropic cases \eqref{eq:b;truncated:qua:ani}, \eqref{eq:L;truncated:qua:ani}, \eqref{eq:gm:b:ani}, \eqref{eq:hl:b:ani} are completely similar and we omit them here. Let's begin with the KL properties of $L_{GY}(u, l)$.

\begin{lemma}
	The isotropic $L_{GY}(u, l)$ is KL-function. 
\end{lemma}
\begin{proof}
	We see $\gph L_{GY}$ can be written as follows with $l=(l_1, l_2) \in \mathbb{R}^{2MN}$. Denoting $x=(u,l_1,l_2,z)$ and $|D(u)_i| := {|(Au-u_0)_i|^2}/{2}$, we have 
	\begin{align}
	&\gph L_{GY}=\big\{x  :z=\sum_{i=1}^{M_0N_0}|D(u)_i| +\frac{\lambda}{2}\sum_{i=1}^{MN}\frac{|\grad u-l|_i^2}{2}+\sum_{i=1}^{MN}H(l_i;\frac{\lambda}{\mu})\big\}=\notag\\
	&\big\{x:z-(\sum_{i=1}^{M_0N_0}|D(u)_i| +\frac{\lambda}{2}\sum_{i=1}^{MN}\frac{|\grad u-l|_i^2}{2}+\sum_{i=1}^{M_0N_0}\sqrt{\frac{\lambda}{\mu}}|l_i|-\frac{|l_i|^2}{2})=0, l_{1i}^2+l_{2i}^2> \frac{\lambda}{\mu})\big\}\notag\\
	&\bigcup\big\{x:z-(\sum_{i=1}^{M_0N_0}|D(u)_i| +\frac{\lambda}{2}\sum_{i=1}^{MN}\frac{|\grad u-l|_i^2}{2}+\frac{\lambda}{\mu}), 0 \leq l_{1i}^2+l_{2i}^2 \leq  \frac{\lambda}{\mu})\big\}. \notag
	\end{align}
	Denote $s:=\sqrt{l_1^2+l_2^2}$, $U_i:= \{ (x,s):l_{1i}^2+l_{2i}^2 = s_i^2, s_i > \sqrt{\frac{\lambda}{\mu}}\}$ and $V_i:= \{ (x,s):l_{1i}^2+l_{2i}^2 = s_i^2, 0\leq s_i \leq \sqrt{ \frac{\lambda}{\mu}}\}$ for $i=1,2, \cdots, M_0N_0$. We found the above representation of $\gph L_{GY}$ can be formulated as
	\begin{align}
	&\big(\big\{(x,s):z-(\sum_{i=1}^{M_0N_0}|D(u)_i|+\frac{\lambda}{2}\sum_{i=1}^{MN}\frac{|\grad u-l|_i^2}{2}+\sum_{i=1}^{M_0N_0}\sqrt{\frac{\lambda}{\mu}}s_i-\frac{s_i^2}{2})=0\big\} \bigcap_{i=1}^{M_0N_0}U_i \big) \notag\\
	&\bigcup \big(\big\{(x,s):z-(\sum_{i=1}^{M_0N_0}|D(u)_i|+\frac{\lambda}{2}\sum_{i=1}^{MN}\frac{|\grad u-l|_i^2}{2}+\frac{\lambda}{\mu})=0\big\}\bigcap_{i=1}^{M_0N_0}V_i\big) \notag 
	\end{align}
	Since all the sets above are semialgebraic sets, with Tarski-Seidenberg Theorem, i.e., Theorem \ref{Tarski-Seidenberg}, $\gph L_{GY}$ is semialgebraic sets.  $ L_{GY}$ is semialgebraic function and is a KL-function. 
\end{proof}

For the cases of the isotropic or anisotropic $L_{GR}(u,b)$, $L_{GM}(u,b)$, and $L_{MS}(u,s)$,  the proofs are quite similar to Lemma \ref{GR_GM_MS:KL} and we omit here.
\begin{lemma}\label{GR_GM_MS:KL}
	The functions $L_{GR}(u,b)$, $L_{GM}(u,b)$, and $L_{MS}(u,s)$ are semialgeraic functions and thus are KL-functions.
\end{lemma}

For the KL property  of $L_{HL}(u, b)$, we need to employ the o-minimal structure. Actually, we have the following lemma. 
\begin{lemma}
	The function $L_{HL}(u, b)$ is definable and  is  a KL-function. 
\end{lemma}
\begin{proof}
	Denoting $b=e^{w}$, we have $\log(b)=w$. 
	Since 
	\[
	\gph  L_{HL}(u, b) = \{(u, b, z):z=\sum_{i=1}^{M_0N_0}\frac{|(Au-u_0)_i|^2}{2}+\sum_{i=1}^{MN}\frac{\mu}{2}(b_i\frac{|(\nabla u)_i|}{\lambda}+b_i- \log(b_i)-1)\},
	\]
	it  thus can be reformulated as 
	\begin{align}
	&\biggl\{(u, b, z, w):z=\sum_{i=1}^{M_0N_0}\frac{|(Au-u_0)_i|^2}{2}+\sum_{i=1}^{MN}\frac{\mu}{2}(e^{w_i}\frac{|(\nabla u)_i|}{\lambda}+e^{w_i}-w_i-1), b_i=e^{w_i}\biggl\} \notag \\
	&=\biggl\{(u, b, z, w):\biggl(  z-\sum_{i=1}^{M_0N_0}\frac{|(Au-u_0)_i|^2}{2}+\sum_{i=1}^{MN}\frac{\mu}{2}(e^{w_i}\frac{|(\nabla u)_i|}{ \lambda}+e^{w_i}-w_i-1) \biggl)^2 \notag \\
	&\qquad \qquad \qquad \qquad  +\sum_{i=1}^{MN}(b_i-e^{w_i})^2=0\biggl\}. \notag
	\end{align}
	The above set on $\{(u,b,z,w)\}$ is a zero set of $P(u, b, z, w, e^u, e^b, e^z, e^w)=0$ with real polynomial function $P$.	Here $e^u=(e^{u_1}, \cdots, e^{u_{MN}})$, $e^b=(e^{b_1}, \cdots, e^{b_{MN}})$, $e^z=(e^{z_1}, \cdots, e^{z_{MN}})$, and $e^w=b$. We thus conclude  the $\gph  L_{HL}(u, b)$ is definable due to Theorem \ref{R_exp:o-minimal}.  Thanks to Theorem \ref{o-minimal:KL}, $L_{HL}(u, b)$ is also  definable and is a KL function. 
\end{proof}

Furthermore, although the analysis of the KL exponent is very challenging, however, for  the KL exponent of the anisotropic $L_{GY}^A(u, l)$, we have the following lemma.
\begin{lemma}\label{lem:ani:GY:KL}
	The anisotropic $L_{GY}^A(u, l)$ is a KL-function with an exponent of $\frac{1}{2}$.
\end{lemma}
\begin{proof}
	Suppose the vectorized $\bds{l}_1$ and $\bds{l}_2$ are $(l_{1,1},\cdots, l_{1,MN})$ and  $(l_{2,1},\cdots, l_{2,MN})$.
	Let's introduce 
	\[
	r_{i,j}(l_{i,j}) =
	\begin{cases}
	\sqrt{\frac{\lambda}{\mu}}|l_{i,j}|-\frac{|l_{i,j}|^2}{2}, &|l_{i,j}|\le \sqrt{\frac{\lambda}{\mu}}, \\
	\frac{\lambda}{2\mu}, &|l_{i,j}|>\sqrt{\frac{\lambda}{\mu}},
	\end{cases}
	\]
	where $i = 1, 2$ and $j=1, \cdots, MN$. 
	We thus rewrite $	L_{GY}^A(u, l)$ as follows
	\begin{align}
	&\sum_{i=1}^{M_0N_0}\frac{|(Au-u_0)_i|^2}{2}+ \sum_{i=1, 2}\sum_{j=1}^{MN}\left(\frac{\mu}{2}{|(\nabla_i u)_j-l_{i,j}|^2}+\mu r_{i,j}(l_{i,j})\right) \label{eq:indicator1} \\
	&\sum_{i=1}^{M_0N_0}\frac{|(Au-u_0)_i|^2}{2} + \sum_{i=1, 2}\sum_{j=1}^{MN}\left(\frac{\mu}{2}{|(\nabla_i u)_j -l_{i,j}|^2} + \mu\min_{k=1, 2}\{r_{i,j,k}(l_{i,j })+\zeta_{i,j,k}(l_{i,j}) \}\right) \notag
	\end{align}
	where $r_{i,j,1}(l_{i,j})=\sqrt{\frac{\lambda}{\mu}}|l_{i,j}|-\frac{|l_{i,j}|^2}{2}$, $r_{i,j,2}(l_{i,j})=\frac{\lambda}{2\mu}$,  $i=1,2, j=1, \cdots, MN$, and
	\begin{equation}\label{eq:constraint:set}
	\zeta_{i,j,1}(l_{i,j}):= I_{\{ l_{i,j}: |l_{i,j}|\leq \sqrt{\frac{\lambda}{\mu}} \}}(l_{i,j}), \quad \zeta_{i,j,2}(l_{i,j}):= I_{\{l_{i,j}: |l_{i,j}| > \sqrt{\frac{\lambda}{\mu}}\}} (l_{i,j}). 
	\end{equation}
	In order to put the piecewise polyhedral terms including $\zeta_{i,j,k}$ and $|l_{i,j}|$ together, introducing
	\[
	\delta_{i,j,1}(l_{i,j}):= \zeta_{i,j,1}(l_{i,j})+ \sqrt{{\lambda}{\mu}}|l_{i,j}|, \quad \delta_{i,j,2}(l_{i,j}):= \zeta_{i,j,2}(l_{i,j}), 
	\]
	with \eqref{eq:indicator1}, we can also reformulate $L_{GY}^A$ as follows
	\begin{equation}
	L_{GY}^A(u,l)=\min_{k=1, 2}\{\mathfrak{F}_{i,j,k}(u, l) + \delta_{i,j,k}(l_{i,j})\},\\	\label{eq:separtesum:1}
	\end{equation}
	where 
	\begin{align}
	&\mathfrak{F}_{i,j,1}(u, l):=\sum_{i=1}^{M_0N_0}\frac{|(Au-u_0)_i|^2}{2} + \frac{\mu}{2}\sum_{i=1, 2}\sum_{j=1}^{MN}\left(\frac{|(\nabla_i u)_j -l_{i,j}|^2}{2}-\frac{|l_{i,j}|^2}{2}\right),  \\
	&\mathfrak{F}_{i,j,2}(u, l):=\sum_{i=1}^{M_0N_0}\frac{|(Au-u_0)_i|^2}{2} +\frac{\mu}{2} \sum_{i=1, 2}\sum_{j=1}^{MN}\left(\frac{|(\nabla_i u)_j -l_{i,j}|^2}{2} + \frac{\lambda}{2}\right).
	\end{align}
	Denote $x:=(u,l)=(u,l_1, l_2)^T$. It can be seen that  $\delta_{i,j,1}(l_{i,j})$ and $\delta_{i,j,2}(l_{i,j})$ are polyhedral functions, since both the constraint sets in \eqref{eq:constraint:set}  are polyhedral sets and $|l_{i,j}|$ is also polyhedral function. Furthermore, it can be readily checked that $\mathfrak{F}_{i,j,k}(u, l)$ can be written as follows,
	\begin{equation}\label{eq:quadratic:symmetric}
	\mathfrak{F}_{i,j,k}(x)= 	\mathfrak{F}_{i,j,k}(u,l_1,l_2)=\frac{1}{2}x^TM_kx + x^Tc_k + b_k, \quad k=1,2, 
	\end{equation}
	where $M_1$ and $M_2$ are symmetric matrices,
	\[
	M_1 = \Sigma^T\Sigma, \quad M_2 = M_1 - \text{Diag}[0, \mu I , \mu I],\quad  \Sigma:= \begin{bmatrix}
	A &0 &0 \\
	\sqrt{\mu} \nabla_1 & -\sqrt{\mu}  I &0 \\
	\sqrt{\mu} \nabla_2 & 0 & -\sqrt{\mu} I
	\end{bmatrix},
	\]
	and
	\[
	c_k=[-A^Tu_0, 0,0]^T, \quad k=1,2, \quad b_1 = \|u_0\|_{2}^2/2,  \quad b_2 = \|u_0\|_{2}^2/2+ {\lambda}/{2}.
	\]
	Since $\mathfrak{F}_{i,j,k}$ can be written in \eqref{eq:quadratic:symmetric} and  $\delta_{i,j,k}$ with $k=1,2$ are polyhedral functions, by \cite{LP} (Corollary 5.2),  we conclude that 
	$L_{GY}^A$ is a KL function of $(u,l_1,l_2)$ with KL exponent $1/2$.
\end{proof}
\section{Global Convergence of the Proposed Algorithms} \label{sec:convergence}
In this section, we will discuss the convergence of  \eqref{eq:general:al:min} that can cover all the models discussed in this paper. We will first prove the boundedness  of the iteration sequence $(u^k,y^k)$ in \eqref{eq:general:al:min} for the corresponding models and the convergence proof then follows.
Actually, from the updates \eqref{update:x} and \eqref{update:y}, we can easily arrive at:
\begin{subequations}\label{update:L_down}
	\begin{align}
	& L(u^k, y^{k-1}) + \frac12 \|u^k-u^{k-1}\|_{M_{k-1}}^2 \leq L(u^{k-1}, y^{k-1}), \\
	& L(u^k, y^k) +\frac12\|y^k-y^{k-1}\|_{N_{k-1}}^2\le L(u^{k}, y^{k-1}), \\
	& L(u^k, y^k)+\frac12 \|u^k-u^{k-1}\|_{M_{k-1}}^2+\frac12\|y^k-y^{k-1}\|_{N_{k-1}}^2\le L(u^{k-1}, y^{k-1}), 
	\end{align}
\end{subequations}
which tells that $L(u^k, y^k)$ is bounded by $0$ and $L(u^0,y^0)$ and is decreasing.

\begin{lemma}\label{lem:bound:gy}
	With $N_k = \mu I$, $M_k = M $ satisfying \eqref{eq:M:N:condition} and $A^*A$ has a bounded inverse, the sequence $\{(u^k, l^k), k\in \mathbb{N}\}$ of the iterations \eqref{eq:general:al:min} for $L_{GY}$ is bounded. 
\end{lemma}
\begin{proof}
	By the condition on $A^*A$, we see $F(u)$ in \eqref{equation:iso} is coercive on $u$. Since, 
	\begin{equation*}
	F(u^k) \leq \inf_{l \in Y} L_{GY}(u^k, l) \leq   L_{GY}(u^k, l^k) \leq L_{GY}(u^0, l^0), 
	\end{equation*}
	we conclude that $u^k$ must be bounded by the coercivity of $F(u)$. It can be verified that $L_{GY}$ is also coercive on $l$ for any fixed $u$, we get the boundedness of $l^k$.
\end{proof}
\begin{lemma}\label{lem:bound:gr}
	With $N_k = \lambda I/2$, $M_k $ satisfying \eqref{eq:M:N:condition} and $A^*A$ has a bound inverse,  the sequence $\{(u^k, b^k), k\in \mathbb{N}\}$ of the iterations \eqref{eq:general:al:min} for $L_{GR}$ is bounded.
\end{lemma}
\begin{proof}
	The boundedness of $\{b^k\}$ follows from the constraint $I_{\{0\leq b \leq 1\}}(b)$ and the projection
	\begin{equation}\label{eq:proj:p}
	b^{k+1} = \mathcal{P}_{[0,1]}(b^k - \frac{\lambda}{\mu}|\nabla u^k|^2+1) \in [0,1],
	\end{equation}
	where $\mathcal{P}_{[0,1]}$ is the projection to $[0,1]$. 
	
	
	By the condition on $A^*A$,  we get the the coercivity of $L_{GR}$   \eqref{equation:iso}. 
	We thus get the boundedness of $u^k$ similarly to Lemma \ref{lem:bound:gy}. 
\end{proof}


Now, let's turn to the boundedness of the iterative sequence of $L_{GM}$.
\begin{lemma}
	With $N_k = \mu I/2$, $0 < b^0 \leq 1$, $M_k $  satisfying \eqref{eq:M:N:condition} 
	and $A^*A$ has a bounded inverse,
	the sequence $\{(u^k, b^k), k\in \mathbb{N}\}$ of the iterations \eqref{eq:general:al:min} for $L_{GM}$ is bounded.
\end{lemma}
\begin{proof}
	We will first prove $0<b^{k+1} \le 1$ by induction on assumption $0<b^{k} \le 1$. The update for the $ b_{k+1}$ is $(\xi+1-\frac{1}{\sqrt{b_{k+1}}} +b_{k+1}-b_k)=0$. With introducing $x:=\sqrt{b_{k+1}}$, we arrive at
	\[
	x^3+px+q=0,  \ \,\  p:=\xi+1-b_k \geq 0, \ q:=-1, \ \xi=\frac{|\nabla u|^2}{\lambda}.
	\]
	Observing that $p>0$, $q<0$, we conclude $ \Delta:=\frac{q^2}{4}+\frac{p^3}{27} \geq \frac{1}{4}$ and $-\frac{1}{2}+\sqrt{\Delta} \geq 0$. Using the celebrated Cardano's formula for the depressed cubic equation, we get
	\begin{equation}\label{eq:update:b:gm}
	x=\sqrt[3]{\frac{1}{2}+\sqrt{\Delta}}+ \sqrt[3]{\frac{1}{2}-\sqrt{\Delta}}=\sqrt[3]{\frac{1}{2}+\sqrt{\Delta}}- \sqrt[3]{-\frac{1}{2}+\sqrt{\Delta}}.
	\end{equation}
	Besides, Noting that
	\[
	\frac{1}{2}+\sqrt{\Delta}=\sqrt[3]{-\frac{1}{2}+\sqrt{\Delta}}^3+\sqrt[3]{1}^3 
	\le(\sqrt[3]{-\frac{1}{2}+\sqrt{\Delta}} +\sqrt[3]{1})^3 \rightarrow \sqrt[3]{ \frac{1}{2}+\sqrt{\Delta}} \leq \sqrt[3]{-\frac{1}{2}+\sqrt{\Delta}} +{1}
	\]
	we obtain that $0<x\le1$, i.e. $0<b_{k+1}=x^2\le 1$.

	By the condition on $A^*A$,  we get the the coercivity of $L_{GM}$ on $u$  in \eqref{eq:gm:original}. 
	The boundedness of $\{u^k\}$ follows similarly to Lemma \ref{lem:bound:gy}. 
\end{proof}

\begin{lemma}
	With $N_k = \mu I/2$, $0\leq b^0 \leq 1$, $M_k $  satisfying \eqref{eq:M:N:condition} 
	and $A^*A$ has a bounded inverse,
	the sequence $\{(u^k, b^k), k\in \mathbb{N}\}$ of the iterations \eqref{eq:general:al:min} for $L_{HL}$ is bounded.
\end{lemma}
\begin{proof}
	Note that 
	$b_{k+1}=\argmin_b \frac{\mu}{2}(\xi b + b-\log b-1) +\frac{\|b-b_k\|_{\mu I}^2}{4}$ with
	$\xi:=\frac{|\nabla u|^2}{\lambda} \geq 0$
	leading to the quadratic equation $b^2 + (\xi +1-b^k)b - 1=0$. Using the quadratic formula and choosing the positive root,  we get 
	\begin{equation}\label{eq:update:b:hl}
	b_{k+1} = (-a+\sqrt{a^2+4})/{2}, \quad a=\xi+1 - b_k \geq0.
	\end{equation}
	By the assumption on $b^0$ and observing that $a^2+4 \le(a+2)^2$, we have $0 < b_{k+1}\le 1$ by deduction. The sequence $\{b^k\}$ is thus bounded. 
	
	By the condition on $A^*A$,  we get the coercivity of $L_{HL}(u,b)$ for $u$ as in \eqref {eq:HL:orignal}. 
	The boundedness of $\{u^k\}$ follows similarly to Lemma \ref{lem:bound:gy}. 
\end{proof}

\begin{lemma}
	The sequence $\{(u^k,s^k), k\in \mathbb{N}\}$ of the iterations \eqref{eq:general:al:min} for $L_{MS}$ is bounded. 
\end{lemma}
\begin{proof}
	It can be seen that although $L_{MS}$ in \eqref{eq:ms:b} is not convex on $(u,s)$, it is however coercive on $(u,s)$, since $L_{MS}(u,s) \rightarrow +\infty$ either $|u|\rightarrow +\infty$ or $|s| \rightarrow +\infty$. Then $(u^k,s^k)$ must be bounded by the boundednes of $L_{MS}(u^k,s^k)$ for all $k$.
\end{proof}
With \eqref{eq:proj:p}, \eqref{eq:update:b:gm}, and \eqref{eq:update:b:hl}, we get the updates of $b^k$ for $L_{GR}$, $L_{GM}$, and $L_{HL}$ correspondingly. Now, let's discuss the update of $l^k$ for $L_{GY}$.
\begin{lemma}\label{new_update_x:GY}
	Assuming $N_k = \kappa_k I$, for the $l_i$ updates of the anisotropic $L_{GY}^A(u,l)$ with $i=1,2$ $\emph{(}$or for the isotropic $L_{GY}(u,l)$ with $i=I$ case$\emph{)}$, we have
	\begin{equation}\label{GY:l_update}
	l_i^{k+1}=\begin{cases}
	\hat{l}^k_i-\sqrt{a}\tau\frac{\hat{l}^k_i}{|\hat{l}^k_i|}\quad&\tau\sqrt{a}<|\hat{l}^k_i|<(1+\tau)\sqrt{a}\\
	0\quad&-\tau\sqrt{a}\le |\hat{l}_i^k|\le\tau\sqrt{a}\\
	\frac{1}{1+\tau}\hat{l}^k_i\quad&|\hat{l}^k_i|>(1+\tau)\sqrt{a}
	\end{cases},
	\end{equation}
	where $a=\frac{\lambda}{\mu},\tau=\frac{\mu}{\kappa_k},\hat{l}^k_i=l_i^{k}+\frac{\mu}{\kappa_k} \nabla_i u^{k+1}$ with  $i=1,2$ for the anisotropic case and the norm $|\cdot|$ denotes the absolute value $\emph{(}$or for the isotropic case  with $i=I$, $l^k_I=({l_1^k}_{I},{l_2^k}_{I})$, $\nabla_I=\nabla$,  and  $|l_{I}|=|(l_1,l_2)|=\sqrt{l_1^2+l_2^2}$$\emph{)}$. 
\end{lemma}
\begin{proof}
	Using \eqref{update:y} and noting $h_{\frac{\lambda}{\mu}}^*(l_i)= H(l_i;\frac{\lambda}{\mu}) -{l_i^2}/{2}+\frac{1}{2} \frac{\lambda}{\mu}$, we obtain that:
	\begin{align*}
	l_i^{k+1} &= \argmin_{l_i}\frac{\mu}{2}\|\nabla_i u^{k+1} -l_i\|^2 + \mu H(l_i;\frac{\lambda}{\mu}) + \frac12\|l_i-l_i^{k}\|_{\kappa_kI}^2 \\
	&=\argmin_{l_i}  h_{a}^*(l_i) + \frac{1}{2\mu}\|l_i-(l_i^{k}+\frac{\mu}{\kappa_k} \nabla_i u^{k+1})\|_{\kappa_k I}^2\\
	&=\argmin_{l_i} h_{a}^*(l_i)+\frac{1}{2\mu}\|l_i-\hat{l}^k_i\|_{\kappa_k I}^2,
	\end{align*}
	where $h_{\tau}^*$ is the same as in \eqref{eq:dual:max:func}. Using the first-order optimal condition, we have 
	\begin{equation}\label{GY:anisotropic:l_update}
	0\in \partial h_a^*(l_i^{k+1})+\frac{1}{\tau}(l_i^{k+1}-\hat{l}^k_i)=\begin{cases}
	\sqrt{a}\frac{l_i^{k+1}}{|l_i^{k+1}|}+\frac{1}{\tau}(l^{k+1}_i-\hat{l}^k_i)\quad&0<|l_i^{k+1}|<\sqrt{a}\\
	B_{\sqrt{a}}(0)-\frac{1}{\tau}\hat{l}^k_i\quad&l_i^{k+1}=0\\
	l^{k+1}_i+\frac{1}{\tau}(l^{k+1}_i-\hat{l}^k_i)\quad&|l^{k+1}_i|\ge\sqrt{a}
	\end{cases},
	\end{equation}
	where $B_{\sqrt{a}}(0)$ is a closed ball with center 0 and radius $\sqrt{a}$ in $\mathbb{R}$ or $\mathbb{R}^2$. Solving this equation, we get \eqref{GY:anisotropic:l_update}.
\end{proof}

Now let's turn to the global convergence of the iteration \eqref{eq:general:al:min} for our models. Actually, with the boundedness of $\{(u^k, y^k)\}$ proved above, by \cite{ABS} (Theorem 6.2), we can get the global convergence.
The following proposition will tell the convergence properties, i.e., Proposition \ref{prop:summable}, whose proof can be obtained by slight modifications of the proof in \cite{ABRC} (Proposition 3.1) and is omitted here. 
\begin{proposition}\label{prop:summable}
	Let $\{(u^k, y^k)\}$ be a sequence generated by our algorithm \eqref{eq:general:al:min}. Denoting $\Gamma$ as the set of accumulation points of the sequence $\{(u^k, y^k)\}$, then the following statements hold:\label{theroem:convergence_property}
	\begin{enumerate}
		\item[\emph{(i)}] 
		$\lim_{k\to\infty}(\|u^k-u^{k-1}\|_{M_{k-1}}^2+\|y^k-y^{k-1}\|_{N_{k-1}}^2)=0$ and it follows from
		\[
		\sum_{k=1}^{\infty}(\|u^k-u^{k-1}\|_{M_{k-1}}^2+\|y^k-y^{k-1}\|_{N_{k-1}}^2)<+\infty;
		\] 
		\item[\emph{(ii)}] if $(u^k, y^k)$ is bounded, then $\Gamma$ is a nonempty compact connected set, and
		\item[\emph{(iii)}] $\Gamma\subset \emph{crit} \ L$. Moreover, $L(u,y)\equiv\zeta$ on $\Gamma$ is constant, where $\zeta=\inf L(u^k, y^k)<+\infty$;
		\[
		\emph{dist}((u^k, y^k), \Gamma)\to 0, \ \text{as} \ k\to\infty,
		\]
		where  $\emph{crit} \ L:=\{ (\bar x,\bar y): 0 \in \partial L(\bar x,\bar y)\}$.
	\end{enumerate}
\end{proposition}
The following Theorem \ref{thm:finite:length} can be found in \cite{ABRC} (see Theorem 4.6 therein), which tells that the iteration sequence $\{(u^k, y^k)\}$ has a finite length  and converges to a critical point.  
\begin{theorem}\label{thm:finite:length}
	Assuming $L$ in \eqref{eq:general:al:min} is a KL-function and $\{(u^k, y^k)\}$ is bounded, we have
	\begin{equation}
	\sum_{k=1}^{+\infty}(\|u^k-u^{k-1}\|_{M_{k-1}}+\|y^k-y^{k-1}\|_{N_{k-1}})<+\infty, \label{equation:convergence}
	\end{equation}
	which means that $(u^k, y^k)$ converges to a critical point of $L$. 
\end{theorem}
Actually, with the global  convergence for convex and proximal alternating optimization given in \cite{Aus, CP0}, our preconditioned framework can also deal with convex half quadratic models.  
We next consider the convergence rate of the sequence $(u^k,y^k)$ for our models if the KL exponent of $L(u,y)$ is known. This kind of convergence rate analysis is standard; see \cite{AB,ABRC,LP} for more comprehensive analysis and is omitted here.
\begin{theorem}[convergence rate]\label{thm:rate}
	Assume that $(u^k,y^k)$ converges to $(\bar{u},\bar{y})$ and $L(u,y)$ has the KL property at $(\bar{u},\bar{y})$ with $\psi(s)=cs^{1-\theta},\theta\in\left[0,1\right)$, $c>0$. Then the following estimations hold:
	\begin{enumerate}
		\item If $\theta=0$ then the sequence $(u^k,y^k)$ converges in a finite number of steps.
		\item If $\theta\in\left(0,\frac12\right]$ then there exist $c>0$ and $\tau\in\left[0,1\right)$, such that $\|(u^k,y^k)-(\bar{u},\bar{y})\|\le c\tau^k$.
		\item If $\theta\in\left(\frac12,1\right]$ then there exists $c>0$, such that $\|(u^k,y^k)-(\bar{u},\bar{y})\|\le c k^{-\frac{1-\theta}{2\theta-1}}$.
	\end{enumerate}
\end{theorem}

With Lemma \ref{lem:ani:GY:KL} and Theorem \ref{thm:rate} we get the following convergence rate for \eqref{eq:L;truncated:qua:ani}.
\begin{corollary}\label{cor:linear}
	Due to Lemma \ref{lem:ani:GY:KL}, the anisotropic $L_{GY}^A(u,l)$ is a KL function of $(u,l)$ with KL exponent $ \frac12$. The sequence $(u^k,l^k)$ generated by preconditioned alternating iterations \ref{eq:general:al:min} thus has linear convergence rate due to Theorem \ref{thm:rate}. 
\end{corollary}

\section{Numerical Part}\label{sec:numer}
In this part, we will first discuss the finite difference method and the preconditioners. Although, the linear equations are in  divergence form with varying coefficients such as the linear subproblems for the $L_{GR}$, $L_{GM}$, $L_{HL}$, and $L_{MS}$, however, we can still benefit from the preconditioning framework by dealing with the linear subproblems with any finite times symmetric Gauss-Seidel iterations. The numerical experiments focused on image denoising and segmentation also show the efficiency of the proconditioned framework. 
\subsection{The Preconditioners and the finite difference  method}

Now let's turn to the discretization of the following equation with $\mathbb{B}$ that is similar in \eqref{eq:bk:general} which covers all the linear equations discussed in this paper
\begin{equation}\label{equation:discrete}
\gamma (x) u + \nabla^* \mathbb{B}(x)\nabla u = z(x),  \quad \mathbb{B}(x)  = \text{Diag}[d^1(x), d^2(x)],  \ x \in \Omega.
\end{equation}
Here we use $d^1$, $d^2$ instead of $b_1$, $b_2$ as in \eqref{eq:bk:general}  for the convenience of lower indices to be employed.
In order to discrete equation \eqref{equation:discrete}, let's first introduce the widely used forward and backward difference operators $\nabla^{+}:=(\nabla_{x}^{+}, \nabla_{y}^{+})^{T}$ and $ \nabla^{-}:=(\nabla_{x}^{-}, \nabla_{y}^{-})^T$  \cite{CP0}
\begin{equation}
\nabla_x^+u_{i,j}=\begin{cases}
u_{i+1,j}-u_{i,j}\quad&1\le i<m,\\
0\quad&i=m,
\end{cases}\quad 
\nabla_y^+u_{i,j}=\begin{cases}
u_{i,j+1}-u_{i,j}\quad&1\le j<n,\\
0\quad&i=n,
\end{cases}
\end{equation}
and 
\begin{equation}
\nabla_x^-u_{i,j}=\begin{cases}
u_{1,j}\quad&i=1,\\
u_{i,j}-u_{i-1,j}\quad&1< i\le m,\\
-u_{m-1,j}\quad&i=m,
\end{cases}\quad 
\nabla_y^-u_{i,j}=\begin{cases}
u_{i,1}\quad&j=1,\\
u_{i,j}-u_{i,j-1}\quad&1< j\le n,\\
-u_{i,n-1}\quad&j=n.
\end{cases}
\end{equation}
If choosing $\nabla^{+}$ as the gradient and its adjoint as the divergence operator, we  get 
\begin{equation}\label{eq:rb:fd}
\gamma_{i,j} u_{i,j}- \left(\nabla_x^{-}(d^1_{i,j}\nabla_x^{+} u_{i,j}) +\nabla_y^{-} (d^2_{i,j}\nabla_y^{+} u_{i,j})\right) =z_{i,j}.
\end{equation}
Written \eqref{eq:rb:fd} in detail with discrete $\nabla^{+}$ and $\nabla^{-}$, taking the interior points for example, we have 
\begin{equation*}
\begin{aligned}
&\big(\gamma_{i,j}+(d_{i,j}^1+d_{i-1,j}^1+d_{i,j}^2+d_{i,j-1}^2)\big)u_{i,j}-\big(d_{i,j}^1 u_{i+1,j}+d_{i-1,j}^1u_{i-1,j}+d_{i,j}^2u_{i,j+1}\\
&+d_{i,j-1}^2u_{i,j-1}\big)=z_{i,j}.
\end{aligned}
\end{equation*}
For the boundary and corner points,  we can also get the corresponding updates and see Tables \ref{tab:laplace-stencil} and \ref{table:fd} for more details. We call this discretization scheme as the normal formula of finite difference (shorten as NFFD).

There is another symmetric finite difference approximation of \eqref{equation:discrete} \cite{WE} (see section 3.4 of \cite{WE}) with the following additional $\widetilde{\nabla}_x^{-},\widetilde{\nabla}_x^{+},\widetilde{\nabla}_y^{-},\widetilde{\nabla}_y^{+}$ defined by 
\begin{equation}
\widetilde{\nabla}_x^-u_{i,j}=\begin{cases}
0\quad&i=1,\\
u_{i,j}-u_{i-1,j}\quad&1<i\le m,
\end{cases}\quad 
\widetilde{\nabla}_y^-u_{i,j}=\begin{cases}
0\quad&j=1,\\
u_{i,j}-u_{i,j-1}\quad&1<j\le n,
\end{cases}
\end{equation}
and 
\begin{equation}
\widetilde{\nabla}_x^+u_{i,j}=\begin{cases}
u_{2,j}\quad&i=1,\\
u_{i+1,j}-u_{i,j}\quad&1< i\le m,\\
-u_{m,j}\quad&i=m,
\end{cases}\quad 
\widetilde{\nabla}_y^+u_{i,j}=\begin{cases}
u_{i,2}\quad&j=1,\\
u_{i,j+1}-u_{i,j}\quad&1< j\le n,\\
-u_{i,n}\quad&j=n.
\end{cases}
\end{equation}
The symmetric finite difference approximation of \eqref{equation:discrete} \cite{WE} reads as follows (see section 3.4 of \cite{WE})
\begin{equation}\label{eq:sfd}
\begin{aligned}
&\gamma_{i,j} u_{i,j}-\frac{1}{2}\big(\nabla_x^{-}(d^1_{i,j}\nabla_x^{+} u_{i,j})+\widetilde{\nabla}_x^{+} (d^1_{i,j}\widetilde{\nabla}_x^{-} u_{i,j})+\nabla_y^{-}(d^2_{i,j}\nabla_y^{+} u_{i,j}) \\
&+\widetilde{\nabla}_y^{+} (d^2_{i,j}\widetilde{\nabla}_y^{-} u_{i,j})\big) =z_{i,j},
\end{aligned}
\end{equation}
Written out in full and for the interior points, we obtain 
\begin{equation}\label{eq:fd_old_full}
\begin{aligned}
&\big(\gamma_{i,j}+(\alpha_{i,j}^1+\alpha_{i-1,j}^1+\alpha_{i,j}^2+\alpha_{i,j-1}^2)\big)u_{i,j}-\big(\alpha_{i,j}^1u_{i+1,j}+\alpha_{i-1,j}^1u_{i-1,j}\\
&+\alpha_{i,j}^2u_{i,j+1}+\alpha_{i,j-1}^2u_{i,j-1}\big)=z_{i,j},
\end{aligned}
\end{equation}
where 
$\alpha_{i, j}^1 = \frac12(d_{i,j}^1+d_{i+1,j}^1)$, $\alpha_{i,j}^2=\frac12(d_{i,j}^2+d_{i,j+1}^2) $ and  see the Table \ref{table:sfd}  for more details. We call this discretization scheme as the symmetric formula of finite difference (shorten as SFFD). Here we mainly focus on the NFFD and SFFD for these nonhomogeneous  elliptic equations in divergence form.  

With the stencils in Tables \ref{table:fd}, \ref{table:sfd} and the corresponding locations in Table \ref{tab:laplace-stencil}, one can do the classical symmetric Gauss-Seidel including the symmetric red-black Gauss-Seidel iteration (shorten as SRBGS henceforth) conveniently \cite{BS2,SA}.  To do SRBGS iteratons, one should mark all points as red or black points first. For example, let  $\Omega_{red}: =\{(i,j) \in \Omega | i+j \ \text{even}  \}$ and $\Omega_{black}: =\{(i,j) \in \Omega | i+j \ \text{odd}  \}$. Then one can first update all the red points with the five-point stencils in Table \ref{table:fd} or \ref{table:sfd} using all the black neighbor points followed by updating all the black points with the five-point stencils in Table \ref{table:fd} or \ref{table:sfd} with the updated red neighbor points. Finally, update all the red points again. Then the one cycle Gauss-Seidel iterations with updating order ``red $\rightarrow $ black $\rightarrow $ red" is finished and it is called one time SRBGS iteration. For more details of the SRBGS, we refer to \cite{BS1} (section 4.1.3). Henceforth, we will use the following notation to denote the $n$ times SRBGS iterations for the equation \eqref{equation:discrete}
with initial iteration value $u^0$
\[
\text{SRBGS}(\gamma(x), d^1, d^2, z, u^0, n),
\]
which can be seen as one n-folds SRBGS preconditioned iteration according to \eqref{eq:preconditioner:multiple} for \eqref{equation:discrete}. Due to the positive definiteness requirement of $M_k$, taking the $k$-th nonlinear preconditioned alternating minimization of the anisotropic $L_{HL}$ model \eqref{eq:hl:b:ani} for example, we actually do the following preconditioned iteration with tiny positive constant $\eta$ and initial iterative value $u^k$
\begin{equation}
u^{k+1}:  =  \text{SRBGS}(1.0 + \eta, b_1^{k}, b_2^{k}, z^k + \eta u^k, u^k, n),
\end{equation}
according to \eqref{eq:pre:iterations:tiny} in Remark \ref{rem:pre:gs}.

Actually, the discretization schemes \eqref{eq:rb:fd} and \eqref{eq:sfd} are both consistent with the Neumann boundary condition as in \eqref{equation:discrete} 
by the following Remark \ref{rem:neumann}. Furthermore, it can be readily checked that while $b_1$ and $b_2$ are both constants, the schemes \eqref{eq:rb:fd} and \eqref{eq:sfd} are equivalent; see the following Remark \ref{rem:constant:same}.
\begin{remark}\label{rem:neumann}
	We  point out that \eqref{eq:rb:fd} and \eqref{eq:sfd} are both with the Neumann boundary condition $\frac{\partial u}{\partial \nu}=0$, where $\nu$ is the outward pointing unit normal vector of $\Omega$.  For example, for the $\mS_{S}$ points in Table \ref{tab:laplace-stencil}, we have $u_{m+1,j}=u_{m,j},1<j<n$. Moreover, with the notation $\Lambda_{m.j}:=\alpha_{m-1,j}^1+\alpha_{m,j}^2+\alpha_{m,j-1}^2$ for convenience,  we have 
	\begin{equation}\label{sfd:ss:full}
	(\gamma_{i,j}+\Lambda_{m,j})u_{m,j}-(\alpha_{m-1,j}^1u_{m-1,j}+\alpha_{m,j}^2u_{m,j+1}+\alpha_{m,j-1}^2u_{m,j-1})=z_{m,j}
	\end{equation}
	in \eqref{eq:fd_old_full} and for the left item in \eqref{eq:sfd}, we have 
	\begin{equation*}
	\begin{aligned}
	&\gamma_{i,j} u_{m,j}-\frac{1}{2}[-d_{m-1,j}^1\nabla_x^+ u_{m-1,j}-d^1_{m,j}\widetilde{\nabla}_x^-u_{m,j}+(d^2_{m,j}\nabla^+_yu_{m,j}-d^2_{m,j-1}\nabla^+_yu_{m,j-1})\\
	&+(d^2_{m,j+1}\widetilde{\nabla}^-_y u_{m,j+1}-d^2_{m,j}\widetilde{\nabla}^-_y u_{m,j})]\\
	&=\gamma_{i,j} u_{m,j}-\frac{1}{2}[-(d^1_{m-1,j}(u_{m,j}-u_{m-1,j}))+(-d^1_{m,j}(u_{m,j}-u_{m-1,j}))+
	(d^2_{m,j}(u_{m,j+1}\\
	&-u_{m,j})-d^2_{m,j-1}(u_{m,j}-u_{m,j-1}))
	+(d^2_{m,j+1}(u_{m,j+1}-u_{m,j})+d^2_{m,j}(u_{m,j}-u_{m,j-1}))]\\
	&=(\gamma_{i,j}+\Lambda_{m,j})u_{m,j}-(\alpha_{m-1,j}^1u_{m-1,j}+\alpha_{m,j}^2u_{m,j+1}+\alpha_{m,j-1}^2u_{m,j-1}),
	\end{aligned}
	\end{equation*}
	which is equal to the left-hand side of \eqref{sfd:ss:full}.
\end{remark}

\begin{remark}\label{rem:constant:same}
	\eqref{eq:rb:fd} and \eqref{eq:sfd} are the same when $d^1$ and $d^2$ both are constants. 
\end{remark}


\begin{table}
	\centering
	\begin{minipage}{0.34\linewidth}
		\begin{tikzpicture}[thick,scale=0.95]
		\draw[<->] (-0.5,3.5) node[below] {$i$}-- (-0.5,4.5)
		-- (0.5,4.5) node[right] {$j$};
		\draw[step=.5cm,black!25!white,dashed] (0,0) grid (4,4);
		\draw[black] (0,0.75) -- (0,0) -- (0.75,0)
		(1.25,0) -- (2.75,0)
		(0,0.5) -- (0.75,0.5)
		(1.25,0.5) -- (2.75,0.5)
		(0.5,0) -- (0.5,0.75)
		(3.25,0) -- (4,0) -- (4,0.75)
		(3.25,0.5) -- (4,0.5)
		(3.5,0) -- (3.5,0.75)
		(1.5,0) -- (1.5,0.75)
		(2,0) -- (2,0.75)
		(2.5,0) -- (2.5,0.75)
		
		(0,3.25) -- (0,4) -- (0.75,4)
		(1.25,4) -- (2.75,4)
		(0,3.5) -- (0.75,3.5)
		(1.25,3.5) -- (2.75,3.5)
		(0.5,4) -- (0.5,3.25)
		(3.25,4) -- (4,4) -- (4,3.25)
		(3.25,3.5) -- (4,3.5)
		(3.5,4) -- (3.5,3.25)
		(1.5,4) -- (1.5,3.25)
		(2,4) -- (2,3.25)
		(2.5,4) -- (2.5,3.25)
		
		(0,1.5) -- (0.75,1.5)
		(1.25,1.5) -- (2.75,1.5)
		(3.25,1.5) -- (4,1.5)
		(0,2) -- (0.75,2)
		(1.25,2) -- (2.75,2)
		(3.25,2) -- (4,2)
		(0,2.5) -- (0.75,2.5)
		(1.25,2.5) -- (2.75,2.5)
		(3.25,2.5) -- (4,2.5)
		
		(0,1.25) -- (0,2.75)
		(0.5,1.25) -- (0.5,2.75)
		(1.5,1.25) -- (1.5,2.75)
		(2,1.25) -- (2,2.75)
		(2.5,1.25) -- (2.5,2.75)
		(3.5,1.25) -- (3.5,2.75)
		(4,1.25) -- (4,2.75);
		
		\draw[fill,red] (0,0) circle [radius=0.05]
		node[below, red] {$\scriptstyle \mS_{SW}$};
		\draw[fill,red] (2,0) circle [radius=0.05]
		node[below, red] {$\scriptstyle \mS_{S}$};
		\draw[fill,red] (4,0) circle [radius=0.05]
		node[below, red] {$\scriptstyle \mS_{SE}$};
		\draw[fill,red] (0,2) circle [radius=0.05]
		node[left, red] {$\scriptstyle \mS_{W}$};
		\draw[fill,red] (2,2) circle [radius=0.05]
		node[below right, red] {$\scriptstyle \mS_{O}$};
		\draw[fill,red] (4,2) circle [radius=0.05]
		node[right, red] {$\scriptstyle \mS_{E}$};
		\draw[fill,red] (0,4) circle [radius=0.05]
		node[above, red] {$\scriptstyle \mS_{NW}$};
		\draw[fill,red] (2,4) circle [radius=0.05]
		node[above, red] {$\scriptstyle \mS_{N}$};
		\draw[fill,red] (4,4) circle [radius=0.05]
		node[above, red] {$\scriptstyle \mS_{NE}$};
		\end{tikzpicture}
	\end{minipage}\quad

	\caption{Finite-difference stencils for \eqref{equation:discrete}
		with homogeneous Neumann boundary conditions.
		The highlighted entry
		denotes the center element and $u$ is assumed to be
		extended by arbitrary values outside $\Omega$.}
	\label{tab:laplace-stencil}
\end{table}

\begin{table}
	\begin{tabular}{c@{\ \ }c@{\ \ }c}
		$\underbrace{\begin{bmatrix}
			\highlight{\gamma_{i,j} + \Sigma_{1,1}} &  d_{1,1}^2 \\
			- d_{1,1}^1 & 
			\end{bmatrix}}_{=\mS_{NW}}$
		&
		$\underbrace{\begin{bmatrix}
			- d_{1,j-1}^2 & \highlight{\gamma_{i,j} + \Sigma_{1,j}} & - d_{1,j}^2 \\
			& - d_{1,j}^1 & 
			\end{bmatrix}}_{=\mS_{N}}$
		&
		$\underbrace{\begin{bmatrix}
			- d_{1,n-1}^2 & \highlight{\gamma_{i,j} + \Sigma_{1,n}} 
			\\
			& - d_{1,n}^1 
			\end{bmatrix}}_{=\mS_{NE}}$
		\\
		\\[-0.75em]
		$\underbrace{\begin{bmatrix}
			- d_{i-1,1}^1 & 
			\\
			\highlight{\gamma_{i,j} + \mu\Sigma_{i,1}} & - d_{i,1}^2 \\
			- d_{i,1}^1 & 
			\end{bmatrix}}_{=\mS_{W}}$
		&
		$\underbrace{\begin{bmatrix}
			& - d_{i-1,j}^1 &
			\\
			- d_{i,j-1}^2 & \highlight{\gamma_{i,j} + \Sigma_{i,j}} & - d_{i,j}^2 \\
			& - d_{i,j}^1 &
			\end{bmatrix}}_{=\mS_{O}}$
		&
		$\underbrace{\begin{bmatrix}
			& - d_{i-1,n}^1 & 
			\\
			- d_{i,n-1}^2 & \highlight{\gamma_{i,j} + \Sigma_{i,n}} & 
			\\
			& - d_{i,n}^1 & 
			\end{bmatrix}}_{=\mS_{E}}$
		\\
		\\[-0.75em]
		$\underbrace{\begin{bmatrix}
			& - d_{m-1,1}^1 & 
			\\
			& \highlight{\gamma_{i,j} + \Sigma_{m,1}} & - d_{m,1}^2 
			\end{bmatrix}}_{=\mS_{SW}}$
		&
		$\underbrace{\begin{bmatrix}
			& - d_{m-1,j}^1 & 
			\\
			- d_{m,j-1}^2 & \highlight{\gamma_{i,j} + \Sigma_{m,j}} & - d_{m,j}^2 
			\end{bmatrix}}_{=\mS_{S}}$
		&
		$\underbrace{\begin{bmatrix}
			& - d_{m-1,n}^1 & 
			\\
			- d_{m,n-1}^2 & \highlight{\gamma_{i,j} + \Sigma_{m,n}} & 
			\end{bmatrix}}_{=\mS_{SE}}$
	\end{tabular}
	\caption{The five-point stencils for NFFD discretization scheme \eqref{eq:rb:fd}. $\Sigma_{i,j}:=d_{i,j-1}^2+d_{i,j}^2+d_{i-1,j}^1+d_{i,j}^1$ for the interior points.  On the corner, $\Sigma_{1,1}:=d_{1,1}^2+d_{1,1}^1$ and other points on the corner are similar. For the interior boundary points, $\Sigma_{i,1}:=d_{i-1,1}^1+d_{i,1}^1+d_{i,1}^2$ and other points on the interior boundary are similar.
	}
	\label{table:fd}
\end{table}

\begin{table}
	\begin{tabular}{c@{\ \ }c@{\ \ }c}
		$\underbrace{\begin{bmatrix}
			\highlight{\gamma_{i,j} + \Sigma_{1,1}} & -\alpha_{1,1}^2 \\
			-\alpha_{1,1}^1 & 
			\end{bmatrix}}_{=\mS_{NW}}$
		&
		$\underbrace{\begin{bmatrix}
			-\alpha_{1,j-1}^2 & \highlight{\gamma_{i,j} + \Sigma_{1,j}} & -\alpha_{1,j}^2 \\
			& -\alpha_{1,j}^1 & 
			\end{bmatrix}}_{=\mS_{N}}$
		&
		$\underbrace{\begin{bmatrix}
			-\alpha_{1,n-1}^2 & \highlight{\gamma_{i,j} + \Sigma_{1,n}} 
			\\
			& -\alpha_{1,n}^1 
			\end{bmatrix}}_{=\mS_{NE}}$
		\\
		\\[-0.75em]
		$\underbrace{\begin{bmatrix}
			-\alpha_{i-1,1}^1 & 
			\\
			\highlight{\gamma_{i,j} + \Sigma_{i,1}} & -\alpha_{i,1}^2 \\
			-\alpha_{i,1}^1 & 
			\end{bmatrix}}_{=\mS_{W}}$
		&
		$\underbrace{\begin{bmatrix}
			& -\alpha_{i-1,j}^1 &
			\\
			-\alpha_{i,j-1}^2 & \highlight{\gamma_{i,j} + \Sigma_{i,j}} & -\alpha_{i,j}^2 \\
			& -\alpha_{i,j}^1 &
			\end{bmatrix}}_{=\mS_{O}}$
		&
		$\underbrace{\begin{bmatrix}
			& -\alpha_{i-1,n}^1 & 
			\\
			-\alpha_{i,n-1}^2 & \highlight{\lambda + \Sigma_{i,n}} & 
			\\
			& -\alpha_{i,n}^1 & 
			\end{bmatrix}}_{=\mS_{E}}$
		\\
		\\[-0.75em]
		$\underbrace{\begin{bmatrix}
			& -\alpha_{m-1,1}^1 & 
			\\
			& \highlight{\gamma_{i,j} + \Sigma_{m,1}} & -\alpha_{m,1}^2 
			\end{bmatrix}}_{=\mS_{SW}}$
		&
		$\underbrace{\begin{bmatrix}
			& -\alpha_{m-1,j}^1 & 
			\\
			-\alpha_{m,j-1}^2 & \highlight{\gamma_{i,j} + \Sigma_{m,j}} & -\alpha_{m,j}^2 
			\end{bmatrix}}_{=\mS_{S}}$
		&
		$\underbrace{\begin{bmatrix}
			& -\alpha_{m-1,n}^1 & 
			\\
			-\alpha_{m,n-1}^2 & \highlight{\gamma_{i,j} + \Sigma_{m,n}} & 
			\end{bmatrix}}_{=\mS_{SE}}$
	\end{tabular}
	\caption{The five-point stencils for SFFD scheme \eqref{eq:sfd}. Here $\alpha_{i,j}^1=\frac12(d_{i,j}^1+d_{i+1,j}^1)$ with $1\le i<m,1\le j\le n$, $\alpha_{i,j}^2=\frac12(d_{i,j}^2+d_{i,j+1}^2)$ with $1\le i\le m,1\le j<n$ and $\Sigma_{i,j}:=\alpha_{i,j-1}^2+\alpha_{i,j}^2+\alpha_{i-1,j}^1+\alpha_{i,j}^1$ on the interior points. On the corner, $\Sigma_{1,1}:=\alpha_{1,1}^2+\alpha_{1,1}^1$ and other points on the corner are similar. On the interior boundary, $\Sigma_{i,1}:=\alpha_{i-1,1}^1+\alpha_{i,1}^1+\alpha_{i,1}^2$ and other points on the interior boundary are similar.
	}\label{table:sfd}
\end{table}

\begin{table}
	\centering
	\begin{tabular}{p{0.15\linewidth}r@{\ }p{0.65\linewidth}}
		\toprule
		\multicolumn{3}{l}{\textbf{Preconditioned alternating minimization for $L_{GY}$ model}}
		\hfill%
		
		\hfill\mbox{}
		\\
		\midrule
		Initialization: &
		\multicolumn{2}{l}{%
			$(u^0, l^0)$, $\lambda>0$, $\mu>0$,  $0<\bar \eta \leq \eta^{+}$, $N_k\equiv \kappa I$}  \\
		&	\multicolumn{2}{l}{with $0< \kappa \leq \kappa^{+}$,  }
		\\
		&
		\multicolumn{2}{l}{%
			$n \geq 1$ inner iterations for SRBGS}
		\\[\medskipamount]
		Iteration: &
		$z^k$ & $ =u_0-\mu\Div l^k $\\
		&$u^{k+1}$ & $=\text{SRBGS}(1.0 + \eta, \mu, \mu, z^k + \eta u^k, u^k, n) $ \hfill 
		\\[\smallskipamount]
		& $l^{k+1}$  & =
		\eqref{GY:l_update} \ for the isotropic or anisotropic case\\
		\bottomrule
	\end{tabular}
	\caption{Preconditioned alternating minimization for the isotropic or anisotropic Geman-Yang model $L_{GY}$ \eqref{eq:L;truncated:qua} or $L_{GY}^A$  \eqref{eq:L;truncated:qua:ani}. The SRBGS iterations are used as preconditioner for dealing with the equations of constant coefficient  for updating $u^{k+1}$.  }
	\label{tab:GY_update}
\end{table}
\begin{table}
	\centering
	\begin{tabular}{p{0.15\linewidth}r@{\ }p{0.65\linewidth}}
		\toprule
		\multicolumn{3}{l}{\textbf{Preconditioned alternating minimization for $L_{GM}$, $L_{GR}$ or $L_{HL}$ model}\ \  }
		\hfill%
		\hfill\mbox{}
		\\
		\midrule
		
		Initialization: &
		\multicolumn{2}{l}{%
			$(u^0, b^0)$ with $0<b^0\le1$, 	$\lambda>0$, $\mu>0$,  $0< \eta \leq  \eta^{+}$,  }
		\\
		&
		\multicolumn{2}{l}{%
			$N_k\equiv \lambda I/2$ for $L_{GR}$,  \ $N_k\equiv \mu I/2 $ for $L_{GM}$ and $L_{HL}$,}			
		\\
		&
		\multicolumn{2}{l}{%
			$n \geq 1$ inner iterations for SRBGS, }
		\\[\medskipamount]
		Iteration: &$u^{k+1}$&$ =\begin{cases}
		\text{for $L_{GR}$ case:}\\
		\text{SRBGS}(1.0 + \eta, \mu d^{1k}, \mu d^{2k}, u_0 + \eta u^k, u^k, n) \quad \\ 
		\text{for $L_{GM}$ or $L_{HL}$ case:}\\
		\text{SRBGS}(1.0 + \eta, \frac{\lambda}{\mu}d^{2k}, \frac{\lambda}{\mu}d^{1k}, u_0 + \eta u^k, u^k, n) \quad
		\end{cases} $ \hfill 
		\\[\smallskipamount]
		&$b^{k+1}$&$=\begin{cases}
		\eqref{eq:proj:p} \ \text{for $L_{GR}$ case}\\
		\eqref{eq:update:b:gm} \ \text{for $L_{GM}$ case} \\
		\eqref{eq:update:b:hl} \ \text{for $L_{HL}$ case}
		\end{cases}$ \\[\smallskipamount]
		\bottomrule
	\end{tabular}
	\caption{Preconditioned alternating minimization for the isotropic or anisotropic Geman-McClure $L_{GM}$ models \eqref{eq:gm:b} or \eqref{eq:gm:b:ani},  the Geman-Reynolds models $L_{GR}$ \eqref{eq:b;truncated:qua} and \eqref{eq:b;truncated:qua:ani} and the Hebert-Leahy models $L_{HL}$ \eqref{eq:hl:b} and \eqref{eq:hl:b:ani}. The SRBGS iterations are used as preconditioner for dealing with  equations of varying coefficients for updating $u^{k+1}$.}
	\label{tab:GM_update}
\end{table}
\begin{table}
	\centering
	\begin{tabular}{p{0.15\linewidth}r@{\ }p{0.65\linewidth}}
		\toprule
		
		\multicolumn{3}{l}{\textbf{Preconditioned alternating minimization for $L_{MS}$ model}\ \  }
		\hfill%
		\hfill\mbox{}
		\\
		\midrule
		
		Initialization: &
		\multicolumn{2}{l}{%
			$(u^0, s^0)$, \  	$\lambda>0$, $\mu>0$, $\epsilon>0$, \ $0 < \eta \leq \eta^{+}$,\ $0 < \gamma \leq \gamma^{+}$,}
		\\
		&
		\multicolumn{2}{l}{%
			$n \geq 1$ inner iterations for symmetric Gauss-Seidel}
		\\[\medskipamount]
		Iteration: &
		$u^{k+1}$ & $= \text{SRBGS}(1.0 + \eta,2\alpha (s^k)^2, 2\alpha (s^k)^2, u_0 + \eta u^k, u^k, n)$ \hfill 
		\\[\smallskipamount]
		& $s^{k+1}$&$=\text{SRBGS}(2\alpha|\nabla u^k|^2+\frac{\lambda}{2\epsilon}+ \gamma, 2\lambda\epsilon,2\lambda\epsilon,\frac{\lambda}{2\epsilon} + \gamma s^k ,s^k, n) $ \\[\smallskipamount]
		\bottomrule
	\end{tabular}
	\caption{Preconditioned alternating minimization  the Ambrosio–Tortorelli approximated Mumford-Shah model $L_{MS}$ \eqref{eq:ms:b}. The preconditioned iterations are used both for dealing with equations of varying coefficients for updating $u^{k+1}$ and $s^{k+1}$.}
	\label{tab:MS_update}
\end{table}
\subsection{Numerical tests: image denoising and segmentation}
In this section, we present the detailed performance of these models. All experiments are performed in Matlab 2019a on a 64-bit PC with an Inter(R) Core(TM) i5-9300HQ CPU (2.40Hz) and 12 GB of RAM. It can be seen from Table \ref{tab:comparison:ani:psnr} that the anisotropic $L_{GM}$, $L_{GR}$, $L_{HL}$, and $L_{GY}$ models are competitive compared to the widely used TV (total variation) model. Especially, the  $L_{HL}$ model with the SFFD has some improvement compared with TV. Furthermore, for most of the images tested, the SFFD performs better than the NFFD. Generally, with SFFD, one can get better PSNR according to Table \ref{tab:comparison:ani:psnr}.

Figure \ref{fig:my_label} shows the effectiveness for image denoising and segmentation of the models discussed. It can be seen that (i.e., images (a)-(l) of Figure \ref{fig:my_label}) all these nonconvex regularizations including $L_{GM}$, $L_{GR}$, $L_{HL}$, and $L_{GY}$ do not have stair-casing effect as TV regularization, no matter for the isotropic or anisotropic case. The segmentation of the well-known $L_{MS}$ model is very appealing with our preconditioning technique; see images (m)-(p) of Figure \ref{fig:my_label}.

Moreover, Figure \ref{HL:pre&nonpre} shows our motivation for introducing the preconditioning. Let's take the $L_{HL}$ model for example. Other models including $L_{GM}$, $L_{GR}$, $L_{MS}$,  and $L_{GY}$ are similar to our observations and we omit the comparisons for compactness. First, it can be seen that one can benefit by getting more lower energy and better PSNR from the proximal terms in \eqref{eq:general:al:min} by solving with CG (conjugate gradient) directly compared to solving the original systems \eqref{eq:alter:mini:orig} without any proximal terms. Furthermore, with the efficient SRBGS preconditioners, the nonlinear iterations  converge much faster while obtaining the high PSNR and low energy more quickly compared to solving the linear system with proximal terms by CG. Although CG solver with low accuracy is also fast, however, there is no convergence guarantee for the whole nonlinear iterations with CG solver. The proposed preconditioned framework can obtain lower energy and better PSNR faster with a convergence guarantee, which is very promising. 

Furthermore, experimentally,  we observed that all the energy functions of  $L_{GM}$, $L_{GR}$, $L_{HL}$, $L_{GY}$, and $L_{GM}$ are decreasing monotonously, which are consistent with the proximal or preconditioned alternating minimization framework \eqref{eq:general:al:min}. We select some representatives as in image (b) of Figure \ref{fig:convergence and energy} for compactness. 

Moreover, for the local linear convergence rate of $L_{GY}^A$ whose KL exponent can be proved to be $1/2$ in Lemma \ref{lem:ani:GY:KL}, the numerical test also shows the asymptotic linear convergence rate as in image (a) of Figure \ref{fig:convergence and energy}, which is also shown theoretically in Corollary \ref{cor:linear}.

\begin{table}
	\centering
	\caption{Comparison for the anisotropic image denoising models with PSNR (peak signal to noise ratio). The noisy images are as follows: Lena1, Lena2 with size 512 $\times$512 and Monarch1, Monarch2 with size 768 $\times$ 512. The usual zero mean Gaussian white noise of variance $\sigma = 0.1$ for Lena1 or Monarch1 and $\sigma = 0.05$ for Lena2 and Monarch2. The parameters for the corresponding models are as follows. For $L_{GM}$ model, we choose $\mu=0.02$, $\lambda=0.05$ for $\sigma = 0.1$ case and $\mu=0.007$, $\lambda=0.004$ for $\sigma = 0.05$ case. For $L_{GR}$ and $L_{GY}$ models, we choose $\mu=3$, $\lambda=0.01$ for $\sigma = 0.1$ case and $\mu=1.5$, $\lambda=0.05$ for $\sigma = 0.05$ case. For $L_{HL}$ model, we choose $\mu=0.005$, $\lambda=0.001$ for $\sigma = 0.1$ case and $\mu=0.002$, $\lambda=0.0005$ for $\sigma = 0.05$ case. For the anisotropic TV model, the regularization parameter $\alpha$ is chosen as the variance of the noise, i.e., $\alpha=\sigma$. The models along with the finite difference schemes that obtained the best PSNR are highlighted in bold face. 
	}
	\begin{tabular}{|c|c|c|c|c|c|c|c|c|}
		\hline
		\multirow{2}{*}{ }	& \multicolumn{2}{c}{$L_{GM}$} & \multicolumn{2}{|c|}{ $L_{GR}$} & \multicolumn{2}{c|}{ $L_{HL}$} & 
		\multicolumn{1}{c|}{ TV} & \multicolumn{1}{c|}{ $L_{GY}$} \\
		\hline 
		& NFFD & SFFD  & NFFD & SFFD  & NFFD & SFFD  & * & * \\
		\hline		Lena1 &29.17 & 29.57& 28.41&29.23 &29.23 &$\bds{30.02}$ & 29.22& 29.31 \\
		\hline
		Monarch1  &28.93 &29.75& 28.48&29.49 &29.31 & $\bds{30.29}$ & 29.14& 29.62\\
		\hline
		Lena2 &32.33 &$\bds{32.94}$ &32.26 &32.47 & 32.47& 32.90& 31.85& 32.36\\
		\hline
		Monarch2  &32.90 & 33.71 &32.91 &33.32 & 33.25&$\bds{33.94}$ & 32.61&33.18 \\
		\hline
	\end{tabular}
	\label{tab:comparison:ani:psnr}
\end{table}

\begin{figure}
	\centering
	\subfloat[Monarch-noise01]{
		\includegraphics[width=0.23\textwidth]{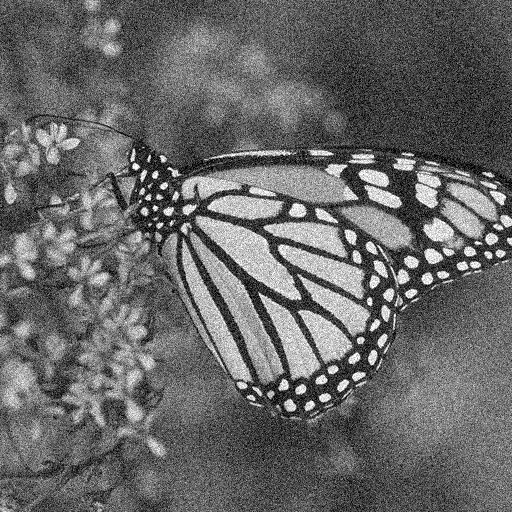}\ \ 
	}
	\subfloat[Monarch-noise005]
	{\includegraphics[width=0.23\textwidth]{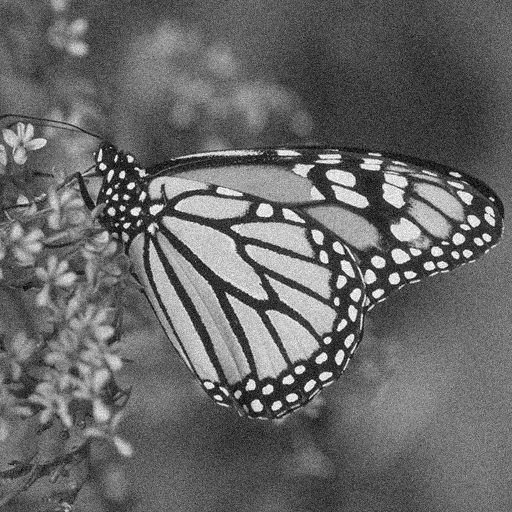}}\ \ 
	\subfloat[Lena-noise01]
	{\includegraphics[width=0.23\textwidth]{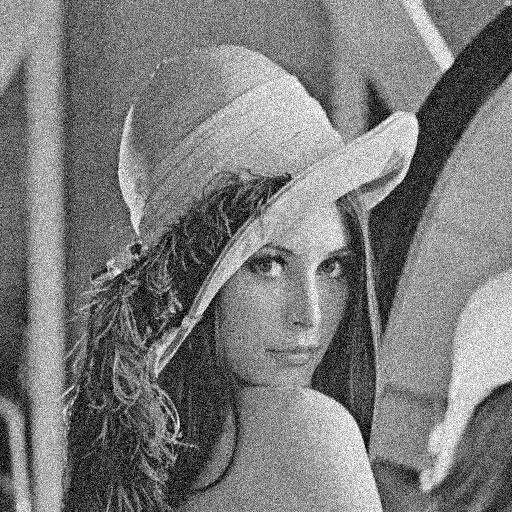}}\ \  
	\subfloat[Lena-noise005]
	{\includegraphics[width=0.23\textwidth]{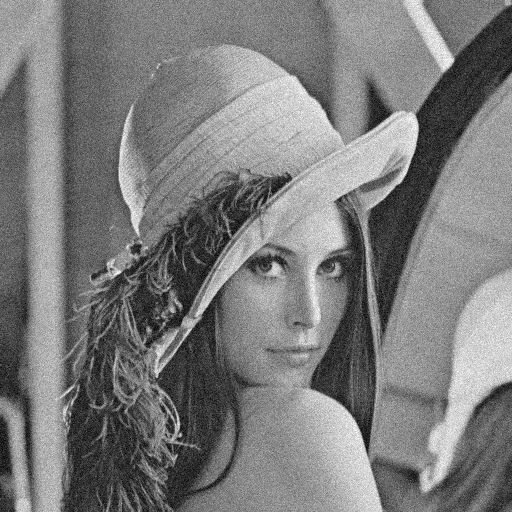}}\\
	\subfloat[Monarch-ani-GM-01]
	{\includegraphics[width=0.23\textwidth]{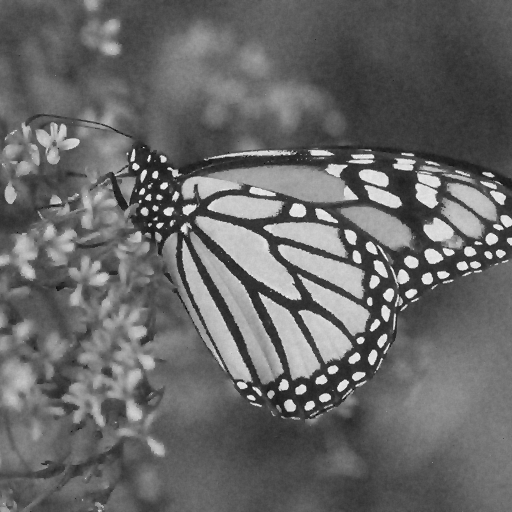}}\ \ 
	\subfloat[Monarch-ani-GR-005]
	{\includegraphics[width=0.23\textwidth]{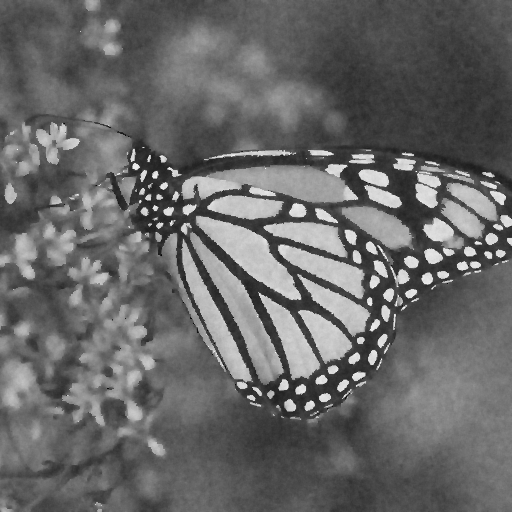}}\ \ 
	\subfloat[Lena-ani-HL-01]
	{\includegraphics[width=0.23\textwidth]{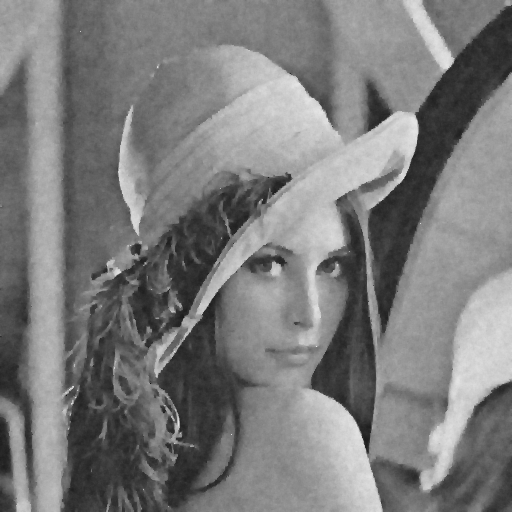}}\ \
	\subfloat[Lena-ani-GY-005]
	{\includegraphics[width=0.23\textwidth]{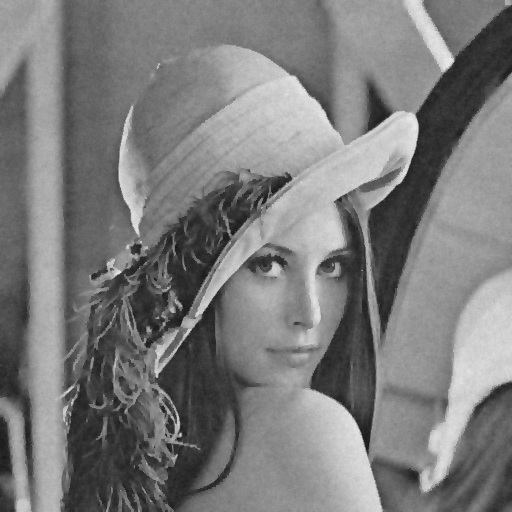}}\\ 
	\subfloat[Monarch-iso-GM-01]
	{\includegraphics[width=0.23\textwidth]{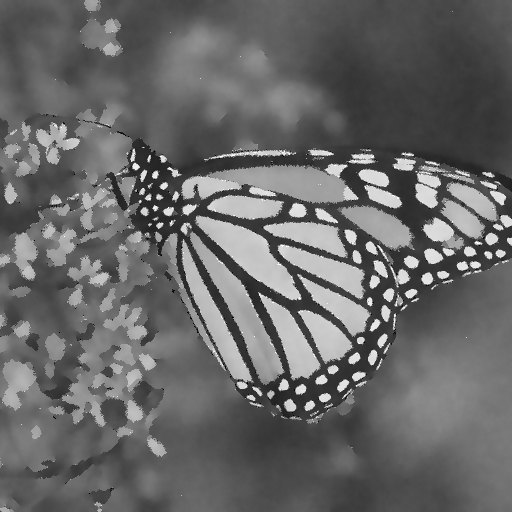}}\ \ 
	\subfloat[Monarch-iso-GR-005]
	{\includegraphics[width=0.23\textwidth]{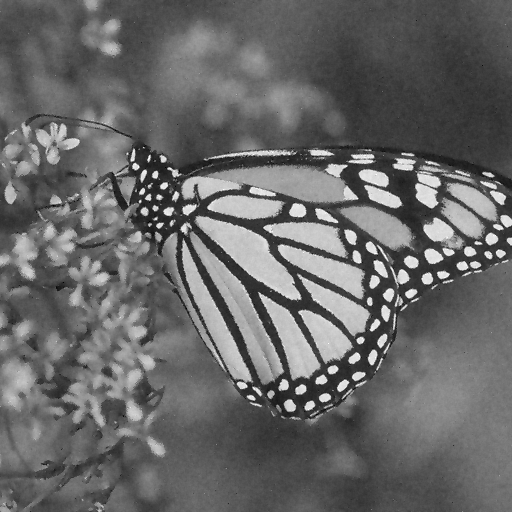}}\ \ 
	\subfloat[Lena-iso-HL-01]
	{\includegraphics[width=0.23\textwidth]{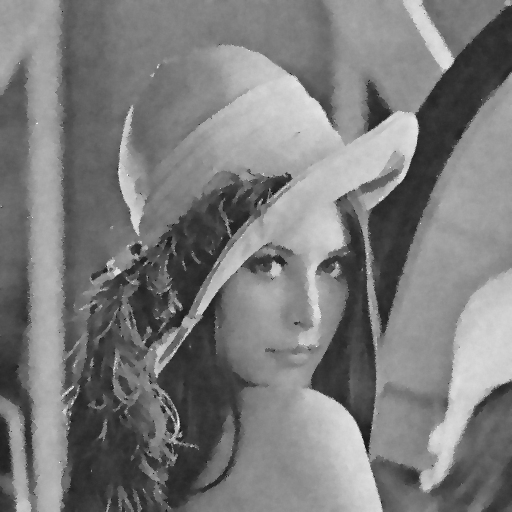}}\ \ 
	\subfloat[Lena-iso-GY-005]
	{\includegraphics[width=0.23\textwidth]{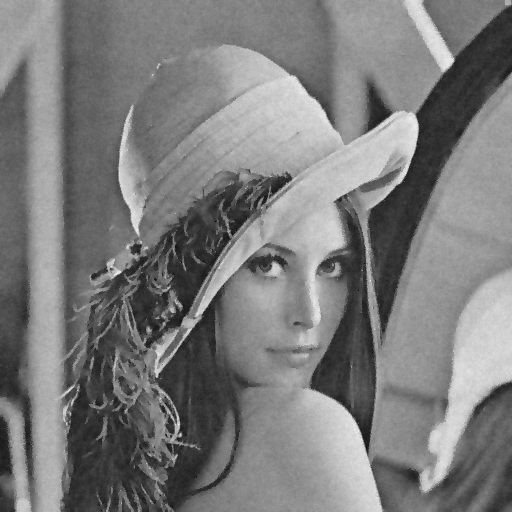}} \\
	\subfloat[Man]
	{\includegraphics[width=0.23\textwidth]{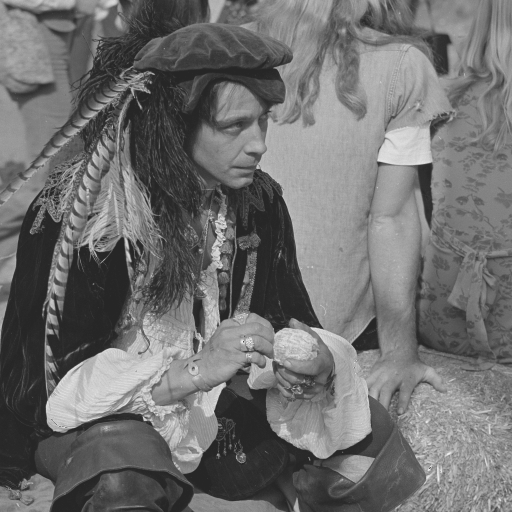}}\ \
	\subfloat[Man-segmentation]
	{\includegraphics[width=0.23\textwidth]{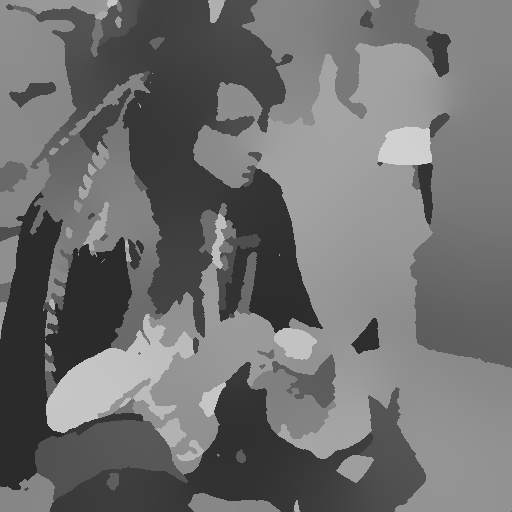}}\ \
	\subfloat[Tulips]
	{\includegraphics[width=0.23\textwidth]{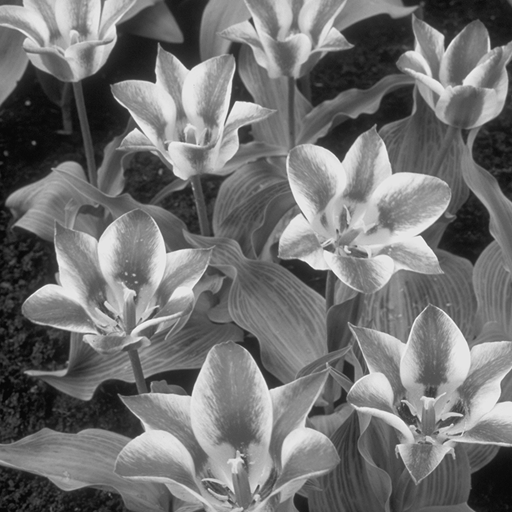}}\ \
	\subfloat[Tulips-segmentation]
	{\includegraphics[width=0.23\textwidth]{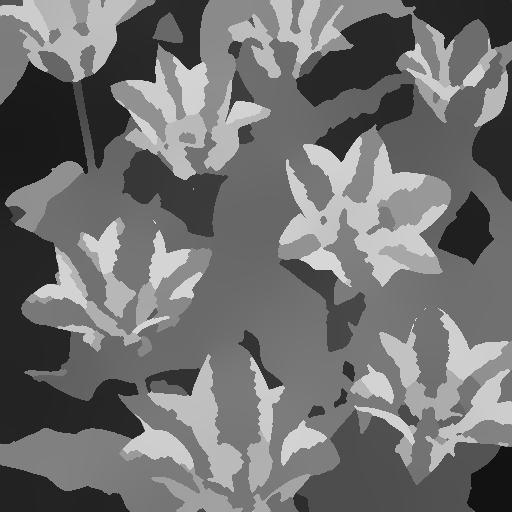}}
	\caption{Image (a) shows the noisy Monarch image of size $512\times512$ corrupted by Gaussian noise of variance $\sigma=0.1$. Image (e) and (i) show the  denoised images of (a) by the anisotropic $L_{GM}$ model with parameters $\mu=0.02$, $\lambda=0.05$ and the isotropic $L_{GM}$ model with parameters $\mu=0.02$, $\lambda=0.001$ correspondingly. Image (b) shows the Monarch noisy image of size $512\times512$ corrupted by Gaussian noise of variance $\sigma=0.0.5$. Image (f) and (j) show the  denoised images of (b) by the anisotropic $L_{GR}$ model with parameters $\mu=1.5$, $\lambda=0.05$ and the isotropic $L_{GR}$ model with parameters $ \mu=1.5$, $\lambda=0.05$ correspondingly. Image (c) shows the corresponding noisy Lena image of size $512\times512$ corrupted by Gaussian noise of variance $\sigma=0.1$. Image (g) and (k) show the denoised images of image (c) by the anisotropic $L_{HL}$ model with parameters $\mu=0.005$, $\lambda=0.001$ and the isotropic $L_{HL}$ model with parameters $\mu=0.005$, $\lambda=0.0005$ correspondingly. Image (d) shows the noisy Lena image of size $512\times512$ corrupted by Gaussian noise of variance $\sigma=0.05$. Image (h) and (l) show the  denoised images of (a) by the anisotropic $L_{GY}$ model with parameters $\mu=1.5$, $\lambda=0.05$ and the isotropic $L_{GY}$ model with parameters $\mu=1.5$, $\lambda=0.005 $ correspondingly. Image (m) shows the original 512$\times$512 Man image and image (n) shows the segmented man image by the MS model $L_{MS}$ with parameters $\alpha=5000$, $\lambda=0.1$ and $\epsilon=0.02$. Image (o) shows the original 512$\times$512 tulips image. Image (p) shows the segmented tulips image by the MS model $L_{MS}$  with parameters $\alpha=3000$, $\lambda=0.1$ and $\epsilon=0.02$.}
	\label{fig:my_label}
\end{figure}

\begin{figure}
	\centering
	\subfloat[ Energy]{\includegraphics[width=0.48\textwidth]{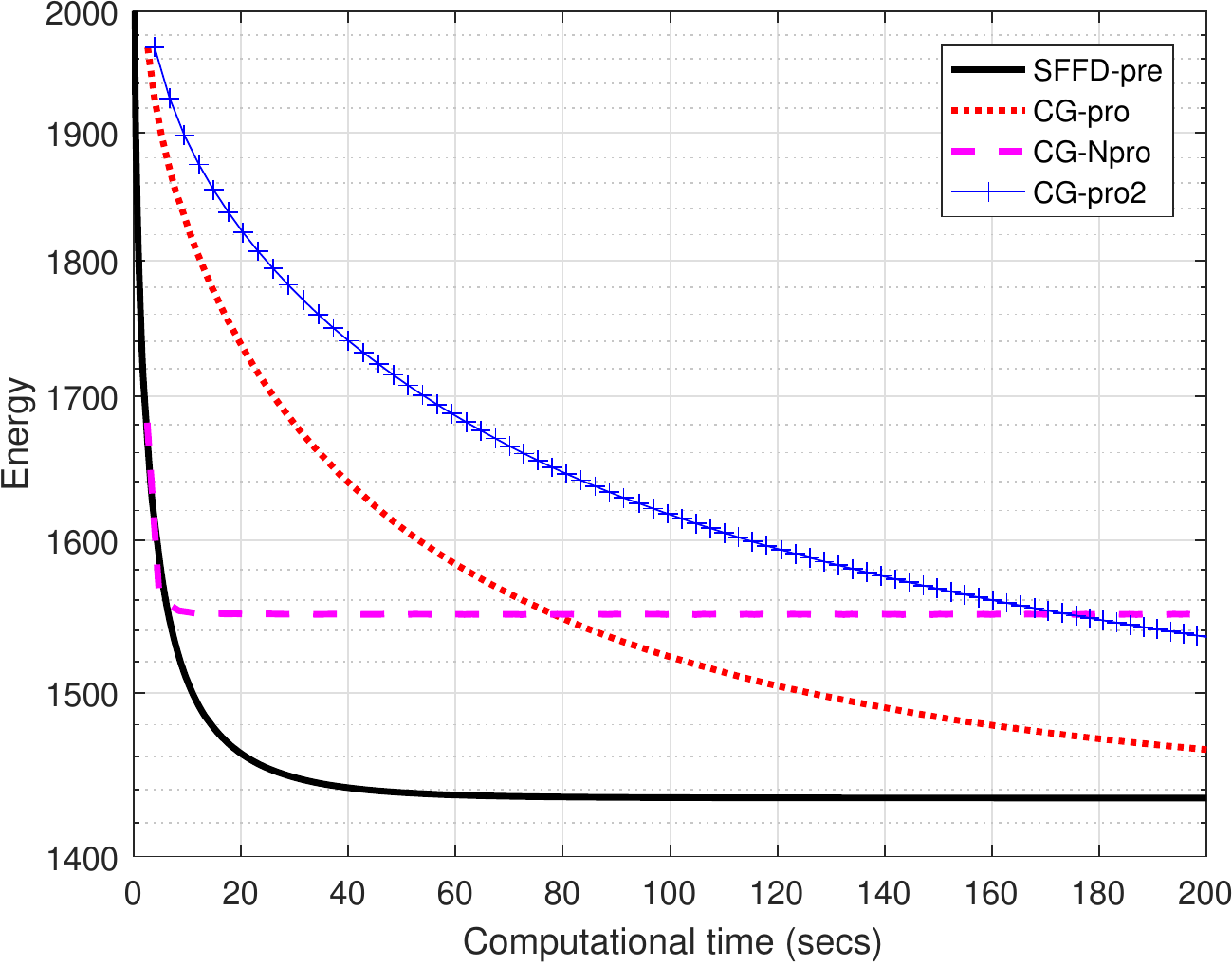}}\ \
	\subfloat[PSNR]{\includegraphics[width=0.48\textwidth]{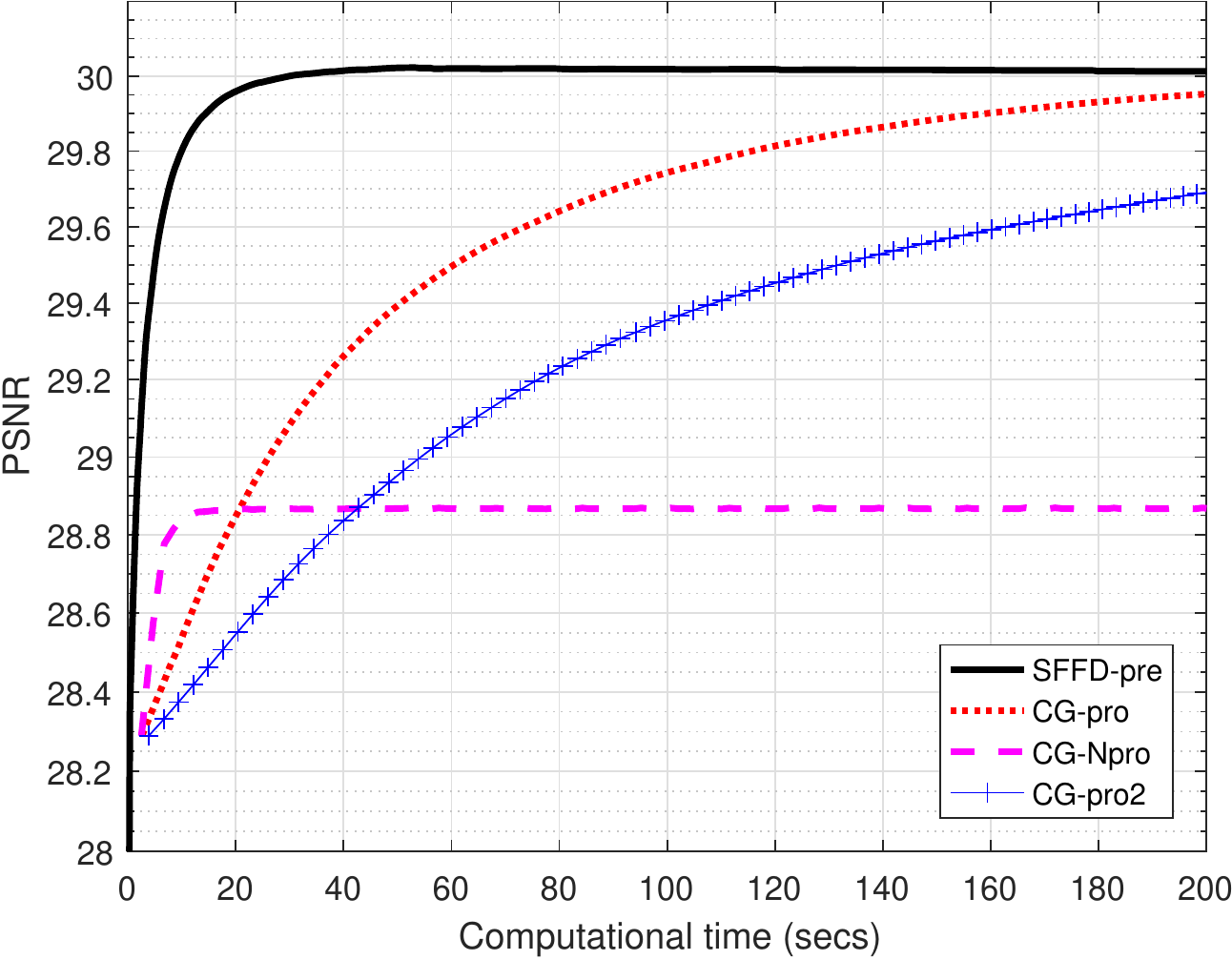}}
	\caption{Figures (a) or (b) shows the comparison with  computational time between preconditioned alternating minimization with 10 times inner symmetric Red-Black Gauss-Seidel (SRBGS) iterations and the conjugate gradient (CG) solver. SFFD-pre denotes the preconditioned alternating minimization by 10 times inner SRBGS preconditioned iterations with SFFD scheme. CG-pro and CG-pro2 denote solving \eqref{update:x} and \eqref{update:y} by CG iterations  with $10^{-3}$ and $10^{-6}$ error stopping criterion correspondingly, while CG-Npro  denotes solving the linear subproblems by CG iterations of the anisotropic $L_{HL}$ model of $10^{-3}$  error stopping criteron  without proximal terms in \eqref{update:x} and \eqref{update:y}. The computation is based on the Lena image of size 512$\times$512 corrupted by Gaussian noise of variance $\sigma=0.1$ for the anistropic $L_{HL}$ model  with parameters$\mu=0.005,\lambda=0.05$. }
	\label{HL:pre&nonpre}
\end{figure}

\begin{figure}
	\centering
	\subfloat[The local linear convergence rate for the anisotropic GY model  $L_{GY}^A$]{\includegraphics[width=0.48\textwidth]{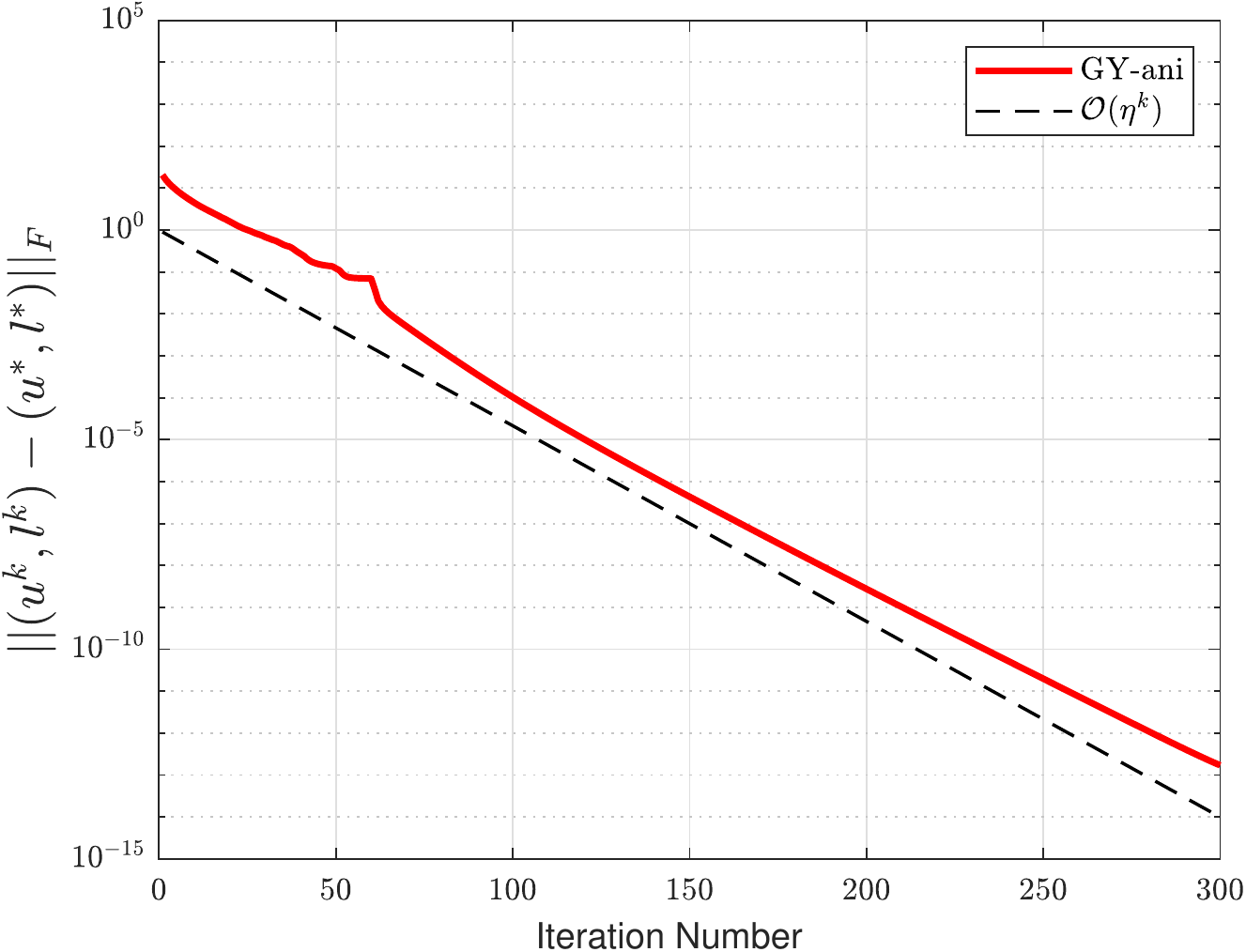}}\ \ 
	\subfloat[Energy descent of some models]{\includegraphics[width=0.48\textwidth]{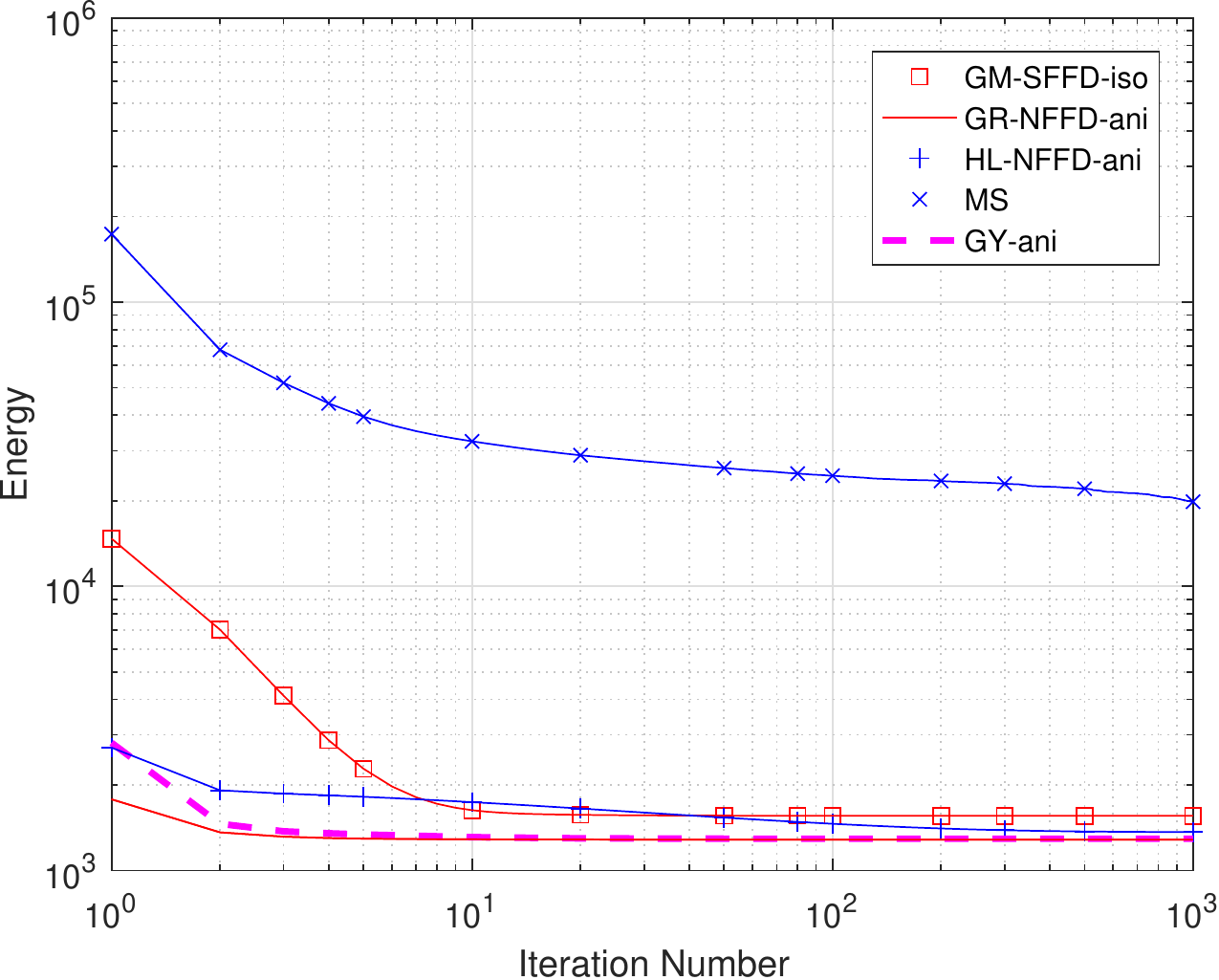}}
	\caption{Figure (a) shows the local linear convergence rate. The computation is by the anistropic GY model $L_{GY}^A$ for the lena image of size 512 $\times$512 with parameters $ \mu = 1.5$, $\lambda = 0.005$ for the Gaussian variance $\sigma=0.05$ cases. Figure (b) shows the energy descent. The computation is by different models for the Monarch image of size $768\times 512$. All models except MS model are used for denoising the noisy Monarch image corrupted by Gaussian noise of variance $\sigma=0.1$ and the MS model $L_{MS}$ is used to segmentation.  The parameters of this models are as follows: $\mu=0.02,\lambda=0.001 $ for the GM-SFFD-iso case (the isotropic $L_{GM}$ model with SFFD), $\mu=1.5,\lambda=0.05 $ for the GR-NFFD-ani case (the anisotropic $L_{GR}$ model with NFFD) and GY-ani cases, $\mu=0.005,\lambda=0.001$ for the HL-NFFD-ani case (the anisotropic $L_{HL}$ model with SFFD) and $\alpha=5000,\lambda=0.1,\epsilon=0.02$ for the MS $L_{MS}$ case. }
	\label{fig:convergence and energy}
\end{figure}



\section{Discussion and Conclusions}\label{sec:conclude}
In this paper, we investigated the proposed framework of preconditioned alternating minimization methods for some typical nonconvex and nonlinear models. With specially designed proximal terms, we can reformulate solving the linear subproblems as the classical preconditioned iterations. We thus can avoid solving the linear subproblems exactly or high accurately. Meanwhile, we can also get the global convergence guarantee. Moreover, we can also obtain high-quality reconstructions including image denoising and segmentation more efficiently as shown in numerics. For the future study, we will consider more general cases including the constraints which are also very important. 

\bibliographystyle{plain}
\bibliography{pEM}

\end{document}